\documentclass[hidelinks,onefignum,onetabnum]{siamart220329}
\usepackage{enumerate}
\usepackage{mathrsfs}
\usepackage{xcolor}
\usepackage{graphicx}
\usepackage{array}
\usepackage{dsfont}
\usepackage{enumitem}
\usepackage{amsfonts}
\usepackage{dsfont}
\usepackage{amssymb}

\let\CAL \mathcal
\let\BBM \mathbb
\let\FRK \mathfrak

\newsiamthm{claim}{Claim}
\newsiamthm{remark}{Remark}
\newsiamthm{assumption}{Assumption}
\newsiamthm{fact}{Fact}
\newsiamthm{example}{Example}

\newcommand\norm[1]{\left\lVert#1\right\rVert}
\newcommand\abs[1]{\left\lvert#1\right\rvert}

\renewcommand{\Re}{\mathbb{R}}
\newcommand\Na{\mathbb{N}}

\renewcommand{\forall}{\text{ for all }}

\newcommand{\ws}{\overset{w*}{\Longrightarrow}}
\newcommand{\sw}{\overset{s}{\longrightarrow}}

\newcommand{\cp}{\overset{\CAL L_p}{\to}}

\newcommand{\ex}[1]{\mathbb{E}\left[#1\right]}

\newcommand\pr[1]{\BBM{P}\left(#1\right)}

\newcommand\ind[1]{\mathds{1}_{\left\{#1\right\}}}
\newcommand{\inner}[1]{\left\langle{#1}\right\rangle}
\newcommand{\interior}[1]{{\kern0pt#1}^{\mathrm{o}}}
\newcommand{\dom}[1]{\text{dom}\left(#1\right)}

\allowbreak
\allowdisplaybreaks

\title{Robustness to Modeling Errors in Risk-Sensitive Markov Decision Problems with Markov Risk Measures}
\author{Shiping Shao \qquad Abhishek Gupta \qquad William B. Haskell}

\begin{document}

\maketitle
\begin{abstract}
We consider risk-sensitive Markov decision processes (MDPs), where the MDP model is influenced by a parameter which takes values in a compact metric space. We identify sufficient conditions under which small perturbations in the model parameters lead to small changes in the optimal value function and optimal policy. We further establish the robustness of the risk-sensitive optimal policies to modeling errors. Implications of the results for data-driven decision-making, decision-making with preference uncertainty, and systems with changing noise distributions are discussed.
\end{abstract}

\section{Introduction}
Risk-sensitive Markov decision processes (MDPs) are an essential paradigm in applications where reliability is a key decision factor.
Risk-sensitivity is often relevant to financial optimization and portfolio planning where the risk is due to extreme market events, and the decision maker (DM) is concerned with more than just expected performance.
Risk-sensitive policies are also frequently deployed in critical infrastructure systems. For example, the electric grid needs to reliably meet random demand in the face of uncertainties due to weather, input prices, and renewable power. Similarly, industrial equipment, vehicles, supply chains, etc. all have to meet functional operating and safety requirements under a wide range of environmental conditions.
In healthcare, planning for patient quality of life outcomes is fundamentally a risk-sensitive problem as well.

To solve any MDP in practice, we need to estimate or construct the model from data.
In addition, in the risk-sensitive paradigm, we need to elicit and input the DM's risk preferences to identify a specific risk-sensitive objective.
In this paper, we capture both of these components via a single key `model parameter' that completely characterizes the risk-sensitive MDP model.
Specifically, it determines the state transition kernel, admissible action set, cost function, and risk-sensitive objective.

There is always fundamentally some modeling error in the choice of this parameter.
First, when estimating the transition kernel and cost function, some statistical uncertainty is introduced.
Second, there is modeling error in the risk-sensitive objective, due to the complexity of dynamic risk models and the difficulty of precisely eliciting the DM's preferences.
Third, the underlying physical system may change and degrade over time due to fatigue and equipment failures, etc. All of these effects may then be expressed as perturbations of the model parameter.

In this paper, we ask the following question: (Q) Under what conditions are the value functions and optimal policies in risk-sensitive MDPs {\it robust to parameter perturbations}? 
Suppose that an optimal risk-sensitive policy has been computed for a nominal parameter value and implemented. While the system is in operation, these parameters can drift leading to a perturbed risk-sensitive MDP. Question (Q) can then be rephrased as: (Q') Under what conditions on the risk-sensitive MDP, do the value functions and optimal policies under the perturbed parameters approximate those for the original nominal one?
Indeed, question (Q') is equivalent to identifying sufficient conditions under which the value functions and optimal policies are continuous in the parameter. Next we present some examples illustrating the practical importance of this question. 

\subsection{Applications with Parametric Uncertainty}
We overview three specific applications that help motivate the problem of sequential optimization under parametric uncertainty, and the issue of sensitivity to parameter perturbations. 

\subsubsection{EV Charging Systems}\label{sec:evch}
Consider a grid aggregator providing EV charging services to the customers in a city. The goal of the aggregator is to maximize operational profits by using renewable energy and scheduling the charging processes. Due to the uncertainties in the renewable energy generation and the charging requests from customers, this problem can be formulated as a risk-sensitive MDP, since the aggregator needs to take the risk of failing to serve demand into account.

These uncertainties are represented as functions or random variables in the system, which are usually parametric. For instance, due to the increase in renewable production over time, the statistics of renewable generation will drift. Due to spreading adoption of EVs, the charging time statistics of the EVs will also drift over time. On the other side, customer demand is affected by the price of charging, traffic, and time, all of which change dynamically. Therefore, in seeking robust strategies for using renewable energy and scheduling the charging processes, it is necessary to consider the impact of changes in these parameters on the risk-sensitive optimal profits and policies.

\subsubsection{Reinforcement Learning}
The DM in reinforcement learning (RL) sequentially evaluates the cost of taking certain actions in certain states. If the system is unknown to the DM, or if it is difficult to formulate an explicit model, then the DM will approximate the system with some parametrized one. For instance, the linear-quadratic-regulator (LQR) problem, which is widely applied in the field of robotics, uses a linear model to approximate the state transition function, and a quadratic model to approximate the cost function. The safe-RL problem is another example, where the DM cannot safely explore the entire state space because some exploration policies may lead to system instability. Thus, safe-RL algorithms only deploy conservative policies, which ensure that the reachable states are within a ``safe set''. Usually, the safe set is represented by a parametric model that can be updated during the training process for policy exploration.

Such RL methods are successful because of the inherent connection between MDPs, RL, and perturbation analysis \cite{cao2003perturbation}. By the robustness property of MDPs, if the approximate model is close enough to the system model, then the DM will arrive at a near optimal policy.

\subsubsection{Preference Uncertainty}
There is an extensive literature on the problem of preference ambiguity in optimization and the difficulty of eliciting the DM's risk preferences. In \cite{armbruster2015decision,delage2018minimizing}, the authors develop robust models for risk-aware optimization where the DM's risk preferences are expressed as an uncertainty set of utility/risk functions. The related stochastic dominance constrained optimization approach is developed in \cite{dentcheva2004optimality}, where the dominance constraints express a requirement for an entire class of risk-sensitive DMs. The problem of preference uncertainty has not yet been studied extensively in the dynamic setting.

\subsection{Related Works}

The theory of risk-sensitive MDP is well-established. Howard and Matheson in \cite{howard1972risk} first incorporated risk sensitivity into an MDP by optimizing the expected exponential utility function of rewards/costs. Jaquette et. al. \cite{jaquette1973markov,jaquette1976utility} investigated MDPs with exponential utility functions and moment optimality, which lexicographically maximize the sequence of signed moments of the total discounted reward. Moreover, Porteus \cite{porteus1975optimality} identified certain conditions where risk-sensitive MDPs can be solved with Bellman equations. 

Other risk criteria have also been applied to the total cost of an MDP, for instance, mean-variance \cite{li2000optimal}, average value-at-risk \cite{bauerle2009dynamic,bauerle2011markov}, target value \cite{wu1999minimizing,boda2004stochastic} that measures the probability of the cost exceeding a target, and general monotone functions \cite{chung1987discounted,bauerle2014more}. In addition, Ruszczyński \cite{ruszczynski2010risk} proposed a dynamic risk measure that sequentially measures the risk of costs in the future with a nested decomposition, and proved that the risk-sensitive MDP can be solved with Bellman equations. We will further introduce the details of these risk-sensitive MDP models in \cref{sec:intro:rs_mdp}. 



Our analysis relies on the theory of continuous parametric MDPs. This theory was first investigated in \cite{maitra1968discounted,furukawa1972markovian} for classical risk-neutral MDPs, and conditions were identified such that the value function is continuous in the state. Stigum \cite{stigum1969competitive} used the continuity of a finite-horizon parametric dynamic programming (DP) problem to prove the existence of a competitive equilibrium in the context of the economy. This result was extended and refined later by Jordan \cite{jordan1977continuity} to establish continuity of the value function with respect to the parameter in an infinite-horizon parametric MDP. Dutta et. al. \cite{dutta1994parametric} also studied continuity of the value function, and relaxed the joint continuity assumption required by \cite{jordan1977continuity} to separate continuity in a parametric MDP with monotone value functions. All of these results are for a risk-neutral DM, whereas in this paper, we study the continuity of the value function and optimal policy for a risk-sensitive DM.

\subsection{Contributions and Outline of this Paper}
Our present work generalizes the continuity results for risk-neutral MDPs in \cite{dutta1994parametric} to risk-sensitive MDPs. Our key contributions are as follows:
\begin{enumerate}
    \item We show that if a parametric risk measure is jointly continuous on its domain and parameter space, then the risk envelop in its biconjugate representation is hemicontinuous with respect to the parameter. This allows us to employ Berge's Maximum Theorem to establish the continuity of the value function of the MDP with Markov risk measures.
    \item We prove that if the cost function, the transition kernel, and the admissible action set are jointly continuous in the state, action, and parameter, then the value function of the risk-sensitive MDP is jointly continuous in the state and the parameter and the optimal policy is lower semicontinuous in the state and parameter. 
    \item We then relax the above joint continuity conditions. We assume separate continuity of the cost function, the transition kernel, and the admissible action set in the state-action pair and the action-parameter pair. We further make some monotonicity assumptions on the MDP, so that the value function is a monotone non-decreasing function of the state. Under these conditions, we establish that the value function remains continuous in the state and parameter. 
    \item Finally, we propose sufficient conditions for the value function to be Lipschitz continuous with respect to the state and parameter. The corresponding Lipschitz coefficients of the value functions are also provided for both infinite and finite-horizon risk-sensitive MDPs. These coefficients explicitly bound the change of the value function in terms of the perturbation in the parameters. We further demonstrate that the policy remains lower semicontinuous in the state and the parameter in this setting. 
\end{enumerate}

This paper is organized as follows: in \cref{sec:problem}, we formulate the risk-sensitive MDP and provide some preliminaries. We also pose our main questions about parametric risk-sensitive MDPs here. In \cref{sec:dutta}, we review the parametric continuity results for risk-neutral MDPs. In \cref{sec:main_res}, we present our main results: the sufficient conditions for the value function of the risk-sensitive MDPs to be continuous (we provide the proofs separately in \cref{sec:proof} for easier readability). In \cref{sec:lip}, we determine sufficient conditions for the value function to be Lipschitz continuous, and identify the Lipschitz coefficients. We then provide some examples to illustrate the joint continuity of the risk measure in \cref{sec:example}. We conclude the paper in \cref{sec:conclusion}.

\subsection{Notation and Definitions}
\subsubsection{Spaces}
Let $(\Omega,\CAL F, \BBM P)$ be a probability space, where $\Omega$ is the set of scenarios, $\CAL F$ is the $\sigma$-algebra of events, and $\BBM P$ is the probability measure on $\CAL F$. We let $\CAL L_p(\Omega,\CAL F, \BBM P)$ denote the space of $p-$integrable random variables, i.e., $\norm{X}_p := \left(\int \abs{X(\omega)}^p\pr{d\omega}\right)^{\frac{1}{p}}<\infty$ for all $X\in\CAL L_p(\Omega,\CAL F,\BBM P)$. For a topological space $\CAL X$, let $\CAL B(\CAL X)$ denote the collection of all Borel measurable subsets, and $\CAL M(\CAL X)$ denote the collection of all probability measures on $\CAL X$.

We let $\bar\Re=\Re\cup\{-\infty,\infty\}$ denote the extended real line. For any function $f:\CAL X\to\bar\Re$, where $\CAL X$ is a normed space, we let $\dom{f}$ denote its domain. We let $\CAL C_b(\CAL X)$ denote the collection of all continuous and bounded functions on $\CAL X$.
Additionally, we let $\CAL C_L(\CAL X)$ denote the collection of all $L-$Lipschitz continuous functions on $\CAL X$ for some $L<\infty$, where $f$ is $L-$Lipschitz if $|f(x) - f(y)| \leq L \|x - y\|$ for all $x, y \in \dom{f}$.

\subsubsection{Ordering}\label{sec:ordering}
When $\CAL X=\Re^d$, we endow it with the usual component-wise order: if $x,x'\in\CAL X$ with $x\leq x'$, then $x_i'\leq x_i'$ for all $i=1,\ldots, d$. For any $X\in\CAL X$, let $X_+:=\max\{0,X\}$. 

Given a random variable $X:\Omega\to\CAL X$, we say $X\sim \mu$ for a probability distribution $\mu:\CAL B(\CAL X)\to[0,1]$ if $\mu(B)=\pr{X\in B}$ for all $B\in\CAL B(\CAL X)$. The cumulative distribution function (CDF) of $X$ is denoted by $F_X(x):=\pr{X\leq x}$ for all $x \in \CAL X$. 

We endow the space of random variables on $\CAL X$ with the {\it first stochastic order} $\preceq$.
For two random variables $X,X':\Omega\to\CAL X$ with $X\sim \mu$ and $X'\sim\mu'$ for $\mu,\mu'\in\CAL M(\CAL X)$, we have $X\preceq X'$ if
\begin{align*}
    \pr{X\geq x}\leq \pr{X'\geq x},\forall x\in\CAL X.
\end{align*}
In this case, we write $\mu\preceq\mu'$.

\subsubsection{Convergence}
Let $\CAL X=\CAL L_p(\Omega,\CAL F,\BBM P)$ for $p\in(1,\infty)$ and $\CAL X^*=\CAL L_p^*(\Omega,\CAL F,\BBM P)$ be the dual space of $\CAL X$. In this paper, the space $\CAL X^*$ is endowed with the weak* topology: We say a sequence of functions $\{g_n\}_{n\in\Na}\subset\CAL X^*$ converges in the weak* sense to $g\in\CAL X^*$, denoted by $g_n\ws g$, if for all  $f\in\CAL X$,
\begin{align*}
    \int f(\omega) g_n(\omega)\pr{d\omega}\to \int f(\omega) g(\omega)\pr{d\omega},\; \text{ as }n\to\infty.
\end{align*}

A sequence of probability measures $\{\mu_n\}_{n\in\Na}\subset\CAL M(\Omega)$ converges in the weak* sense to $\mu\in\CAL M(\Omega)$, denoted by $\mu_n\ws \mu$, if for all $f\in\CAL C_b(\Omega)$,
\begin{align*}
    \int f(\omega) \mu_n(d\omega) \to \int f(\omega) \mu(d\omega),\; \text{ as }n\to\infty.
\end{align*}
We say that $\{\mu_n\}_{n\in\Na}\subset\CAL M(\Omega)$ converges to $\mu\in\CAL M(\Omega)$ setwise, denoted by $\mu_n\sw \mu$, if the above convergence holds for all measurable and bounded functions $f\in\CAL L_\infty(\Omega,\CAL F,\BBM P)$.

Let $\CAL Y$ be a metric space. A transition kernel $q:\CAL B(\CAL X)\times\CAL Y\to[0,1]$ is weakly continuous if $q(\cdot,y_n)\ws q(\cdot,y)$ for all sequences $\{y_n\}_{n\in\Na}$ with $y_n\to y$. Further, we say $q$ is setwise continuous if for all measurable and bounded functions $f : \CAL X \rightarrow \Re$,
\begin{align*}
    \int f(x) q(dx,y_n)\to \int f(x) q(dx,y), \text{ as }n\to\infty.
\end{align*}
In this case, we denote $q(\cdot,y_n)\sw q(\cdot,y)$.





A sequence of random variables $\{X_n\}_{n\in\Na}$ converges to $X$ in $\CAL L_p$, denoted by $X_n\cp X$, if $\norm{X_n-X}_p \to 0$.

\subsubsection{Correspondences}
Let $\CAL A$ be a metric space and $\CAL B$ be a Hausdorff topological space. A correspondence $\Psi:\CAL A \rightrightarrows \CAL B$ is a set-valued map such that $\Psi(a)\subset\CAL B$ for all $a\in\CAL A$. A correspondence is closed-valued (or compact-valued) if $\Psi(a)$ is closed (or compact) in $\CAL B$ for every $a\in\CAL A$.

We next recall the definition of upper/lower hemicontinuity of $\Psi$ from \cite{Aliprantis2006}. 
A closed-valued correspondence $\Psi:\CAL A\rightrightarrows \CAL B$ is upper hemicontinuous at $a\in\CAL A$ if and only if for any sequence $\{a_n\}_{n\in\Na}\subset \dom{\Psi}$, and any sequence $\{b_n\}_{n\in\Na}$ with $b_n\in\Psi(a_n)$, we have that $a_n\to a\in\CAL A$ and $b_n\to b\in\CAL B$ implies $b\in\Psi(a)$. A correspondence $\Psi:\CAL A\rightrightarrows \CAL B$ is lower hemicontinuous at $a\in\CAL A$ if and only if for any $b\in\Psi(a)$ and for any sequence $\{a_n\}_{n\in\Na}\subset \dom{\Psi}$ with $a_n\to a$, there exists a sequence $\{b_n\}_{n\in\Na}$ and $b_n\in\Psi(a_n)$ for all $n\in\Na$ such that $b_n\to b$. A correspondence is continuous if it is both upper and lower hemicontinuous at all points $a\in\CAL A$. 


\section{Problem Formulation}\label{sec:problem}

In this section, we define {\it parametric} MDPs, where all of the model information is expressed by a model parameter.
The model parameter, denoted $\theta\in\Theta$ where $\Theta$ is a compact metric space, describes the cost function, transition kernel, admissible action set, and the DM's risk preferences.



The underlying probability space is $(\Omega,\CAL F, \BBM P)$, the state space is $\CAL S$, and the action space is $\CAL A$. The state and action spaces are assumed to be Borel subsets of {\it Euclidean spaces}. The MDP can be either finite-horizon with time index $\{1,\ldots,T\}$ for $T<\infty$ or infinite-horizon. The initial state $S_0=s_0$ is fixed. We have a filtration $\CAL F_0\subset\CAL F_1\subset \cdots \subset \CAL F$, where $\CAL F_t$ is the $\sigma$-algebra generated by the random variables $\{S_0,A_0,\ldots,S_t\}$ (where $S_t$ and $A_t$ are the random state and action at time $t$). 

We write the dynamics of the system at time $t$ as a Borel measurable function $q_t:\CAL B(\CAL S)\times\CAL S\times\CAL A\times\Theta\to [0,1]$. That is, under parameter $\theta$ and state-action pair $(s_t,a_t)\in\CAL S\times\CAL A$, $q_t(\cdot|s_t,a_t,\theta)$ is a probability measure on $\CAL S$, i.e.,
\begin{align}\label{eq:q}
    \pr{S_{t+1}\in B\middle |s_t,a_t,\theta}=q_t(B|s_t,a_t,\theta),\, \forall B\in\CAL B(\CAL S).
\end{align}
We denote this measure succinctly as $q_t(s_t,a_t,\theta)$, so \cref{eq:q} yields $S_{t+1}\sim q_t(s_t,a_t,\theta)$. The cost function at time $t$ is $c_t:\CAL S\times\CAL A\times\Theta\to[0,\infty)$, and $c_T:\CAL S\times\Theta\to[0,\infty)$ is the terminal cost function (which does not depend on the action) for the finite-horizon case. 

An MDP is said to be stationary if $q_t\equiv q$ and $c_t\equiv c$ for all time $t\in\Na$. If an MDP is infinite-horizon, we assume that it is stationary.
Let $\gamma:\Theta\to[0,\bar\gamma]$ be the discount factor of future costs in the infinite-horizon MDP, where we assume that $\bar\gamma < 1$.

The set of admissible actions in state $s\in\CAL S$ with parameter $\theta$ is given by a correspondence $\Gamma:\CAL S\times\Theta\to \CAL B(\CAL A)$. We let $\CAL D(\theta):=\{(s,a)\in\CAL S\times\CAL A:a\in\Gamma(s,\theta)\}$ denote the set of all feasible state-action pairs for parameter $\theta$.
For each time $t$, the DM picks a map $\pi_t:\CAL S\times\Theta\to \CAL A$ with $\pi_t(s,\theta)\in\Gamma(s,\theta)$ for all $(s, \theta) \in \CAL S \times \Theta$. Then $\pi:=(\pi_0,\pi_1,\ldots)\in\Pi$ denotes a policy for the MDP, where $\Pi$ is the space of all feasible policies. A policy is said to be stationary if $\pi_t = \pi_{t'}$ for all $t\neq t'\in\Na$. Under the parameter $\theta\in\Theta$, the DM selects a policy $\pi\in\Pi$ and faces the sequence of costs:
\begin{align*}
    Z_T(s_0,\theta;\pi)& := (c_0(s_0,\pi_0(s_0,\theta),\theta),\ldots,c_T(s_T,\theta)) \text{ for a finite-horizon};\\
    Z_\infty(s_0,\theta;\pi)& := (c_0(s_0,\pi_0(s_0,\theta),\theta),c_1(s_1,\pi_1(s_1,\theta),\theta),\ldots) \text{ for an infinite-horizon},
\end{align*}
where we suppress the dependence on the underlying scenario for simplicity (i.e., $Z_T(\cdot|s_0,\theta;\pi) : \Omega\to\Re^{T+1}$ and $Z_\infty(\cdot|s_0,\theta;\pi) : \Omega\to\Re^{\infty}$ are mappings from the underlying probability space to cost sequences).

\subsection{Risk-Neutral Problem}

 The risk-neutral finite-horizon performance criteria is the expected total cost
$$
J_{\texttt{RN}, T}(s_0,\theta;\pi) := \ex{\sum_{t=0}^{T-1} c_t(S_t,\pi_t(S_t,\theta),\theta) + c_T(S_T,\theta)},
$$
and the risk-neutral finite-horizon MDP is
$$
\FRK P_{\texttt{RN}, T} : \quad \min_{\pi \in \Pi} J_{\texttt{RN}, T}(s_0,\theta;\pi),
$$
with optimal policy $\pi_{\texttt{RN},T}^*=\arg\min_{\pi\in\Pi} J_{\texttt{RN},T}(s_0,\theta;\pi)$.
The (stationary) infinite-horizon performance criteria is the expected discounted total cost
$$
J_{\texttt{RN}, \infty}(s_0,\theta;\pi) := \ex{\sum_{t=0}^\infty \gamma(\theta)^t c(S_t,\pi(S_t,\theta),\theta)}
$$
and the risk-neutral infinite-horizon MDP is
$$
\FRK P_{\texttt{RN}, \infty} : \quad \min_{\pi \in \Pi} J_{\texttt{RN}, \infty}(s_0,\theta;\pi),
$$
with optimal policy $\pi_{\texttt{RN},\infty}^*=\arg\min_{\pi\in\Pi} J_{\texttt{RN},\infty} (s_0,\theta;\pi)$.

\subsection{Risk-Sensitive Problem}\label{sec:formu:seq}
\label{sec:intro:rs_mdp}

We now consider risk-sensitive MDPs.
To begin, we formalize the notion of a risk measure.
Let $\CAL X = \mathcal L_p(\Omega,\CAL F,\mathbb{P})$ for $1 < p < \infty$ be an admissible space of random variables. We have a risk measure $\rho_\theta:\CAL X\to\bar\Re$ for each value of the parameter $\theta$. A risk measure $\rho_\theta$ is {\it coherent} if it satisfies the following conditions, which were first introduced in \cite{artzner1999coherent}.


\begin{definition}[Coherent Risk Measures]\label{def:coherent}
A risk measure $\rho_\theta$ is coherent if it satisfies:
\begin{enumerate}[label=(\roman*)]
    \item {\bf Monotonicity: }If $X,X'\in\CAL X$ and $X(\omega)\leq  X'(\omega)$ for all $\omega\in\Omega$, then $\rho_\theta(X)\leq\rho_\theta(X')$.
    \item {\bf Convexity: } If $X,X'\in\CAL X$ and $\alpha\in[0,1]$, then $\rho_\theta(\alpha X + (1-\alpha)X')\leq \alpha\rho_\theta(X)+(1-\alpha)\rho_\theta(X')$.
    \item {\bf Translation equivalence: }If $\alpha\in\Re$ and $X\in\CAL X$, then $\rho_\theta(X+\alpha)=\rho_\theta+\alpha$.
    \item {\bf Positive homogeneity: }If $a>0$ and $X\in\CAL X$, then $\rho_\theta(\alpha  X)=\alpha\rho_\theta(X)$.
\end{enumerate}
\end{definition}

Let $\CAL L_t:=\CAL L_p(\Omega,\CAL F_t,\BBM P)$ for all $t\in\Na$, and let $\{\rho_{\theta,t}\}_{t\in\Na}$ be a sequence of {\it one-step conditional risk measures} \cite{ruszczynski2010risk} where each $\rho_{\theta,t}:\CAL L_{t+1}\to\CAL L_t$.
We also suppose all $\{\rho_{\theta,t}\}_{t\in\Na}$ are coherent as in \cref{def:coherent}.
For the finite-horizon case, the risk-sensitive objective is:
\begin{align}
    J_T(s_0,\theta;\pi) := & c_0(s_0,\pi_0(s_0,\theta),\theta) + \rho_{\theta,1}\Big(c_1(S_1,\pi_1(S_1,\theta),\theta)+ \cdots \label{eq:f_ra_X}\\
    &\; + \rho_{\theta,T-1}\Big(c_{T-1}(S_{T-1},\pi_{T-1}(S_{T-1},\theta),\theta)  + \rho_{\theta,T}(c_T(S_T,\theta))\Big) \cdots \Big). \nonumber
\end{align}
This objective is a risk measure on finite sequences $\varrho_\theta:\prod_{t=0}^T \CAL L_t\to\Re$ constructed by composing the one-step risk measures.
The corresponding risk-sensitive MDP is:
$$
\FRK P_T:\quad \min_{\pi \in\Pi} J_T(s_0,\theta;\pi),
$$
with optimal policy $\pi^*=\arg\min_{\pi\in\Pi} J_T(s_0,\theta;\pi)$. Let $\{v_t\}_{t=0}^T$, where $v_t:\CAL S\times\Theta\to\Re$, be the value functions for $\FRK P_T$. Similarly, for the infinite-horizon case, the risk-sensitive objective is:
\begin{align}
    J_\infty(s_0,\theta;\pi) := & c_0(s_0,\pi(s_0,\theta),\theta) + \rho_{\theta,1}\Big(\gamma(\theta) c(S_1,\pi(S_1,\theta),\theta) \label{eq:i_ra_X}\\
    &\qquad +\rho_{\theta,2}\Big(\gamma(\theta)^2 c(S_2,\pi(S_2,\theta),\theta)+\cdots \Big)\Big) \nonumber.
\end{align}
This objective is a risk measure on infinite sequences $\varrho_\theta:\prod_{t=0}^\infty \CAL L_t\to\Re$, which is well-defined by \cite[Theorem 3]{ruszczynski2010risk} under mild assumptions. The corresponding risk-sensitive MDP is:
$$
\FRK P_\infty:\quad \min_{\pi\in\Pi} J_\infty(s_0,\theta;\pi),
$$
with optimal stationary policy denoted by $\pi_\infty^*=\arg\min_{\pi\in\Pi} J_\infty(s_0,\theta;\pi)$. Let $v:\CAL S \times \Theta \to\Re$ be the value function for $\FRK P_\infty$ (in the stationary case).

\subsection{Perturbation of Risk-Sensitive MDPs}

The goal of the risk-sensitive MDP is to obtain the optimal value function and optimal policy ($\{v_t\}_{t=0}^T$ and $\pi^*$ for $\FRK P_T$ or $v$ and $\pi_\infty^*$ for $\FRK P_\infty$). The main objective of this paper is to establish the continuity properties of the value function and policy as a function of $\theta$. In particular, suppose $\{\theta_n\}_{n\in\Na}$ converges to $\theta$, then we ask under what conditions: 
\begin{enumerate}
    \item[Q1.] Does $\min_{\pi\in\Pi}J(s_0,\theta_n;\pi)$ converge to $\min_{\pi\in\Pi}J(s_0,\theta;\pi)$ as $n\to\infty$?
    \item[Q2.] Does the optimal policy $\pi^*(s_0,\theta_n)$ converge to $\pi^*(s_0,\theta)$ as $n\to\infty$?
\end{enumerate}

As stated in the Introduction, the above questions are frequently encountered in market design, control of safety-critical systems, and distributional reinforcement learning. In particular, if the answer to Q1 is affirmative, then the value function of the risk-sensitive MDP under the nominal parameter $\theta$ is ``close'' to the value function under the perturbed parameter $\theta'$, when $\theta$ and $\theta'$ are close to each other. In addition, if the answer to Q2 is affirmative, then the respective optimal policies are also close. Thus, the DM can ignore minor perturbations of the model parameter and not recompute the value functions and optimal policies every time the parameter drifts. Indeed, in practice initial control policies are often designed at the time of manufacturing/installation but then not tuned for the rest of the system lifetime, even though the system components degrade and the underlying distributions of the operating conditions change over the lifetime.




\section{Results for the Risk-Neutral Case}\label{sec:dutta}

We briefly review the existing continuity results for $\FRK P_{\text{RN}, T}$ and $\FRK P_{\text{RN}, \infty}$ from \cite{jordan1977continuity,dutta1994parametric}.
Our goal is to derive analogous results for the risk-sensitive MDPs $\FRK P_T$ and $\FRK P_\infty$. Under mild assumptions on the risk-neutral MDP, the value functions $\{v_t(\cdot,\theta)\}_{t=0}^T$ and $v(\cdot,\theta)$ exist for all $\theta\in\Theta$; see, for example, \cite{hernandez2012discrete,stokey1989recursive,hinderer2005lipschitz}. We now recall conditions under which the value functions of an MDP are continuous, see \cite{dutta1994parametric}.

\begin{assumption}[Jointly Continuous MDP]\label{asm:cts_v_neu}
For all $t\in\Na$:
\begin{enumerate}[label=(\roman*)]
    \item $q_t$ is weak* continuous on $\CAL S\times\CAL A\times\Theta$\label{asm:cts_v_neu:q}.
    \item $c_t$ is continuous on $\CAL S\times\CAL A\times\Theta$ and bounded.\label{asm:cts_v_neu:c}
    \item $\Gamma: \CAL S\times\Theta\rightrightarrows \CAL A$ is continuous and is a compact-valued correspondence.\label{asm:cts_v_neu:A}
    \item $\gamma:\Theta\to(0,\bar\gamma]$ is continuous and $\bar\gamma<1$. \label{asm:cts_v_neu:gamma}
\end{enumerate}
\end{assumption}
\cref{asm:cts_v_neu} requires joint continuity with respect to the state, action, and parameter for all system components: cost functions, transition kernels, admissible action sets, and discount factor.

\begin{theorem}[\cite{dutta1994parametric}, Theorem 1]\label{thm:dutta}
Suppose \cref{asm:cts_v_neu} holds.

(i) The value functions $\{v_t\}_{t=0}^T$ are continuous on $\CAL S\times\Theta$, and $\pi_{\texttt{RN}, T}^*(\cdot,\cdot)$ is lower semi-continuous on $\CAL S\times\Theta$.

(ii) The value function $v$ is continuous on $\CAL S\times\Theta$, and $\pi_{\texttt{RN}, \infty}^*(\cdot,\cdot)$ is lower semi-continuous on $\CAL S\times\Theta$.
\end{theorem}
\cite[Theorem 1]{dutta1994parametric} does not establish continuity for the finite-horizon case. However, one can readily adopt the proof technique of \cite[Theorem 1]{dutta1994parametric} to arrive at the continuity result for finite-horizon MDPs by essentially the same argument.

Next we identify regularity assumptions for the class of monotone MDPs. Recall $\CAL S$ is equipped with the element-wise order (see \cref{sec:ordering}).

\begin{assumption}[Monotone MDP]\label{asm:mono}
For every $t\in\Na$ and every $s, s' \in \CAL S$ such that $s\leq s'$, we have 
\begin{enumerate}[label=(\roman*)]
    \item $q_t(s, a, \theta) \preceq q_t(s', a, \theta)$ for all $(a,\theta)\in\CAL A\times\Theta$.\label{asm:mono:q}
    \item $c_t(s, a,\theta) \leq c_t(s', a,\theta)$ for all $(a,\theta)\in\CAL A\times\Theta$.\label{asm:mono:c}
    \item $\Gamma(s,\theta) \supseteq \Gamma(s',\theta)$ for all $\theta\in\Theta$.\label{asm:mono:A}
\end{enumerate}
\end{assumption}
\noindent
Under \cref{asm:mono}, \cite[Chapter 9]{stokey1989recursive} shows that the value function of a risk-neutral MDP is monotonically increasing  (a simpler proof is presented in \cite[Theorem 5]{li2021fitted}). In the following, we appeal to weaker separate continuity assumptions for monotone MDPs compared to \cref{asm:cts_v_neu}.

\begin{assumption}[Separately Continuous MDP]\label{asm:scts_v_neu}
For every $t\in\Na$:
\begin{enumerate}[label=(\roman*)]
    \item $q_t(\cdot,\cdot,\theta)$ is weak* continuous on $\CAL S\times\CAL A$ for every $\theta\in\Theta$, and $q_t(s,\cdot,\cdot)$ is weak* continuous on $\CAL A\times\Theta$ for every $s\in\CAL S$.\label{asm:scts_v_neu:q}
    \item $c_t(\cdot,\cdot,\theta)$ is continuous on $\CAL S\times\CAL A$ for every $\theta\in\Theta$, and $c_t(s,\cdot,\cdot)$ is continuous on $\CAL A\times\Theta$ for every $s\in\CAL S$.\label{asm:scts_v_neu:c}
    \item $\Gamma(\cdot,\theta)$ is continuous on $\CAL S$ for every $\theta\in\Theta$, and $\Gamma(s,\cdot)$ is continuous on $\Theta$ for every $s\in\CAL S$.\label{asm:scts_v_neu:A}
    \item $\gamma:\Theta\to(0,\bar\gamma]$ is continuous and $\bar\gamma<1$. \label{asm:scts_v_neu:gamma}
\end{enumerate}
\end{assumption}

\begin{theorem}[\cite{dutta1994parametric}, Theorem 3]\label{thm:dutta:mono}
Suppose \cref{asm:mono} and \cref{asm:scts_v_neu} hold. Then:

(i) $\min_{\pi\in\Pi}J_{\texttt{RN}, \infty}(\cdot,\cdot;\pi)$ is continuous on $\CAL S\times\Theta$, and $\pi_{\texttt{RN}, \infty}^*(\cdot,\cdot)$ is lower semi-continuous on $\CAL S\times\Theta$.

(ii) $\min_{\pi\in\Pi}J_{\texttt{RN}, T}(\cdot,\cdot;\pi)$ is continuous on $\CAL S\times\Theta$, and $\pi_{\texttt{RN}, T}^*(\cdot,\cdot)$ is lower semi-continuous on $\CAL S\times\Theta$.
\end{theorem}

\section{Main Results}
\label{sec:main_res}
Although the desired parametric continuity results have been established for risk-neutral MDPs in \cite{dutta1994parametric}, the risk-sensitive extension remains challenging. In the risk-neutral case, continuity of the value function follows directly from weak continuity of the transition kernel (since it is based on an expectation). In the risk-sensitive case, $\rho_\theta$ can be any coherent risk measure, and these have widely differing forms.

We first need a universal representation of coherent risk measures to discuss their continuity properties. The Fenchel-Moreau Theorem establishes that every {\it law invariant}\footnote{$\rho_\theta$ is law invariant if two random variables $X,X'\in\CAL X$, $\rho_\theta(X)=\rho_\theta(X')$ if $F_{X}(u)=F_{X'}(u)$ for all $u\in\Re$. A detailed discussion of law invariant risk measures is given in \cite{shapiro2013kusuoka}}, {\it proper}\footnote{$\rho_\theta$ is proper if $\rho_\theta(X)>-\infty$ for all $x\in\CAL X$ and its domain $\text{dom}(\rho_\theta):=\{x\in\CAL X:\rho_\theta(X)<\infty\}\neq \emptyset$.}, and coherent risk measure $\rho_\theta$ can be represented as its biconjugate as follows.
Let $\CAL X^*=\CAL L_p^*(\Omega,\CAL F,\BBM P)$ be the dual space of $\CAL X$ endowed with the weak* topology. Then let
\begin{align*}
    \CAL P_\Omega:=\left\{\phi\in\CAL X^*:\int \phi(\omega)\pr{d\omega}=1,\;\phi\geq 0\right\}
\end{align*}
be the collection of probability density functions with respect to $\BBM P$ on $\Omega$. Every element in $\CAL P_\Omega$ can be identified with a probability measure on $(\Omega, \CAL F)$, which features $\phi$ as its Radon-Nikodyn derivative (or density) with respect to $\BBM P$.

The robust representation of $\rho_\theta$ is then:
\begin{align}\label{eq:rho_inner}
    \rho_\theta(X):=\sup_{\phi\in\Phi(\theta)}\inner{X,\phi}=&\sup_{\phi\in\Phi(\theta)}\int_{\Omega}X(\omega)\phi(\omega)\pr{d\omega},
\end{align}
where $\Phi(\theta) \subset \CAL P_\Omega$ is the {\it risk envelope}, and we write $\Phi:\Theta\rightrightarrows \CAL P_\Omega$ as a correspondence to emphasize the dependence on $\theta$.
According to \cite{rockafellar2002deviation}, the risk envelope is explicitly:
\begin{align}\label{eq:Phi}
    \Phi(\theta) &=\left\{\phi\in\CAL P_\Omega:\inner{\phi,X}\leq \rho_\theta(X)\forall X\in\CAL X\right\}\nonumber\\
    & =\bigcap_{X\in\CAL X} \left\{\phi\in\CAL P_\Omega:\inner{\phi,X}\leq \rho_\theta(X)\right\} \subset \CAL P_\Omega.
\end{align}
Representation \cref{eq:rho_inner} essentially amounts to taking the supremum of expectations of the value function over a set of ``tilted'' distributions.
We must establish the relationship between the continuity of the risk-sensitive value function and the continuity of the risk envelope.

For risk-sensitive MDPs, solving the DP decomposition requires some form of continuity (indeed, lower hemicontinuity) of the risk envelope. However, this assumption is not an obvious condition even if the one-step risk measures are continuous. In addition, the risk envelope $\Phi$ is parameterized by the state, action, transition kernel, and parameter. Continuity of the supremum of the integrals in \cref{eq:rho_inner} will follow from the Berge Maximum Theorem.
This proof technique requires us to establish the continuity of the integral and continuity of the risk envelope $\Phi$ in \cref{eq:rho_inner}.
To show that the integral in \cref{eq:rho_inner} is continuous, we need to appeal to Lebesgue's dominated convergence theorem with varying measures.

\subsection{Jointly Continuous MDPs}

Continuity of sequential risk measures requires additional conditions on the one-step risk measures $\rho_{\theta,t}$. Our main result leverages the class of Markov risk measures, which was first studied by Ruszczynski et. al. \cite[Definition 6]{ruszczynski2010risk}.

\begin{definition}[Markov Conditional Risk Measure and Risk Transition Mapping]\label{def:markov}
Let $\CAL V_{\CAL S}:=\CAL L_p(\CAL S,\CAL B(\CAL S),\BBM P)$. A sequence of risk measures $\{\rho_{\theta,t}\}_{t\in\Na}$, where $\rho_{\theta,t} : \CAL V_\CAL S \rightarrow \Re$, is Markov with respect to $\{s_t\}_{t\in\Na}$ under the following conditions. For any $v(\cdot,\theta)\in \CAL V_{\CAL S}$ and $\theta\in\Theta$, there exists a mapping $\sigma_t:\CAL V_{\CAL S}\times\CAL S\times\CAL P_{\CAL S}\times\Theta\to\Re$ such that
\begin{align}\label{eq:rho_sigma}
    \rho_{\theta,t}(v(\cdot,\theta))=\sigma_t(v(\cdot,\theta), s_t,q_t(s_t,a_t,\theta),\theta),
\end{align}
where $\{\sigma_t\}_{t\in\Na}$ satisfies for every $t\in\Na$:
\begin{enumerate}
    \item For all $(s,a)\in\CAL D(\theta)$, the mapping 
    \begin{align*}
        v(\cdot,\theta)\mapsto\sigma_t(v(\cdot,\theta),s,q_t(s,a,\theta),\theta)
    \end{align*}
    is a coherent risk measure on $\CAL V_{\CAL S}$.
    \item For all $v(\cdot,\theta)\in\CAL V_{\CAL S}$ and every policy $\pi(\cdot,\theta)$ measurable on $\CAL S$, the mapping
    \begin{align*}
        s\mapsto \sigma_t(v(\cdot,\theta),s,q_t(s,\pi(s,\theta),\theta),\theta)
    \end{align*}
    is an element of $\CAL V_{\CAL S}$.
\end{enumerate}
Under these conditions, $\{\sigma_t\}_{t\in\Na}$ are referred to as risk transition mappings. For infinite-horizon risk-sensitive MDP, the risk mappings $\sigma_t$ are {\it stationary} (i.e., $\sigma_t\equiv\sigma$ for all $t \in \Na$).
\end{definition}

By \cite[Theorem 2.2]{ruszczynski2006optimization} (see also, \cite{ruszczynski2014erratum}), each $\sigma_t$ has the form
\begin{align}\label{eqn:Phirtm}
\sigma_t(v(\cdot,\theta),s,q_t(s,a,\theta),\theta)=\sup_{\phi\in\Phi_t(s,q_t(s,a,\theta),\theta)}\inner{v(\cdot,\theta),\phi},
\end{align}
where $\Phi_t(s,q_t(s,a,\theta),\theta) \subset \CAL L_p^*(\CAL S,\CAL B(\CAL S),\mathbb{Q}_t)$. This  aligns with  the representation \cref{eq:rho_inner} and captures the dependence on the state $s$, transition kernel $q_t(s, a, \theta)$, and parameter $\theta$. For the (stationary) infinite-horizon case, we have $\BBM Q_t\equiv \BBM Q$ and $\Phi_t\equiv\Phi$ for some $\Phi$ such that
\begin{align}\label{eq:sigma}
    \sigma(v(\cdot,\theta),s,q(s,a,\theta),\theta)=\sup_{\phi\in\Phi(s,q(s,a,\theta),\theta)}\inner{v(\cdot,\theta),\phi}.
\end{align}
We now present the key assumption on the parametric risk sensitive MDP to have continuous value functions.  

\begin{assumption}[Jointly Continuous MDP]\label{asm:cts_v}
For all $t\in\Na$:
\begin{enumerate}[label=(\roman*)]
    \item $q_t$ is setwise continuous on $\CAL S\times\CAL A\times\Theta$ and there exists a measure $\mathbb{Q}_t\in\CAL M(\CAL S)$ and a measurable function $m_t:\CAL S\times\CAL S\times\CAL A\times\Theta\to [0,\infty)$ such that $q_t(ds'|s,a,\theta) = m_t(s';s,a,\theta) \mathbb{Q}_t(ds')$ and $m_t(\cdot;s,a,\theta)\in \CAL L_p^*(\CAL S,\CAL B(\CAL S),\mathbb{Q}_t)$.    
    \label{asm:cts_v:q}
    \item $c_t$ is continuous on $\CAL S\times\CAL A\times\Theta$ and bounded.\label{asm:cts_v:c}
    \item $\Gamma: \CAL S\times\Theta\rightrightarrows \CAL A$ is continuous and is a compact-valued correspondence.\label{asm:cts_v:A}
    \item $\gamma:\Theta\to(0,\bar\gamma]$ is continuous and $\bar\gamma<1$. \label{asm:cts_v:gamma}
\end{enumerate}
\end{assumption}

Next we present our main result for MDPs satisfying the joint continuity conditions given in \cref{asm:cts_v}. 

\begin{theorem}\label{thm:mdp_cts}
Suppose \cref{asm:cts_v} holds. In addition, suppose the conditional risk measures $\{\rho_{\theta,t}\}_{t\in\Na}$ are Markov, coherent, and the risk envelopes $\Phi_t$ are jointly continuous on $\CAL S\times\CAL M(\CAL S)\times\Theta$.

(i) Then, $\min_{\pi\in\Pi}J_T(\cdot,\cdot;\pi)$ is continuous on $\CAL S\times\Theta$ and $\pi_T^*(\cdot,\cdot)$ is lower semi-continuous on $\CAL S\times\Theta$.

(ii) Suppose in addition that the conditional risk measures $\{\rho_{\theta,t}\}_{t\in\Na}$ are stationary.
Then, $\min_{\pi\in\Pi}J_\infty(\cdot,\cdot;\pi)$ is continuous on $\CAL S\times\Theta$ and $\pi_\infty^*(\cdot,\cdot)$ are lower semi-continuous on $\CAL S\times\Theta$.
\end{theorem}

\subsection{Separately Continuous Monotone MDPs}
Next we weaken the joint continuity condition in \cref{asm:cts_v} to separate continuity for monotone MDPs (see \cref{asm:mono} for the requirements on monotone MDPs and \cref{asm:scts_v} for the separate continuity requirements). We recall that if $s_t,s_t'\in\CAL S$ with $S_{t+1}\sim q(s_t,a,\theta)$ and $S_{t+1}'\sim q(s_t',a,\theta)$, then $q(s_t,a,\theta)\preceq q(s_t',a,\theta)$ implies $S_{t+1}\preceq S_{t+1}'$.

\begin{assumption}[Separately Continuous MDP]\label{asm:scts_v}
For every $t\in\Na$:
\begin{enumerate}[label=(\roman*)]
    \item $q_t(\cdot,\cdot,\theta)$ is setwise continuous on $\CAL S\times\CAL A$ for any $\theta\in\Theta$, and $q_t(s,\cdot,\cdot)$ is setwise continuous on $\CAL A\times\Theta$ for any $s\in\CAL S$.\label{asm:scts_v:q}
    \item $c_t(\cdot,\cdot,\theta)$ is continuous on $\CAL S\times\CAL A$ for any $\theta\in\Theta$, and $c_t(s,\cdot,\cdot)$ is continuous on $\CAL A\times\Theta$ for any $s\in\CAL S$.\label{asm:scts_v:c}
    \item $\Gamma(\cdot,\theta)$ is continuous on $\CAL S$ for any $\theta\in\Theta$, and $\Gamma(s,\cdot)$ is continuous on $\Theta$ for any $s\in\CAL S$.\label{asm:scts_v:A}
    \item $\gamma:\Theta\to(0,\bar\gamma]$ is continuous and $\bar\gamma<1$. \label{asm:scts_v:gamma}
\end{enumerate}
\end{assumption}

\begin{theorem}\label{thm:mono_cts}
Suppose \cref{asm:mono} and \cref{asm:scts_v} hold. Also suppose the conditional risk measures $\{\rho_{\theta,t}\}_{t\in\Na}$ are Markov, coherent, and such that $\Phi_t$ is jointly continuous on $\CAL S\times\CAL M(\CAL S)\times\Theta$.

(i) Then, $\min_{\pi\in\Pi}J_T(\cdot,\cdot;\pi)$ is continuous on $\CAL S\times\Theta$ and $\pi_T^*(\cdot,\cdot)$ is lower semi-continuous on $\CAL S\times\Theta$.

(ii) Suppose in addition that the conditional risk measures $\{\rho_{\theta,t}\}_{t\in\Na}$ are stationary.
Then, $\min_{\pi\in\Pi}J_\infty(\cdot,\cdot;\pi)$ is continuous on $\CAL S\times\Theta$ and $\pi_\infty^*(\cdot,\cdot)$ are lower semi-continuous on $\CAL S\times\Theta$.
\end{theorem}

In Section \ref{sec:example}, we identify the conditions on the $\Phi_t$ such that $\Phi_t$ is a continuous correspondence. This yields continuity of $\sigma_t$ in the value functions, state, action, and the parameter by an application of Berge's maximum theorem \cite{Aliprantis2006}.

\subsection{Lipschitz MDPs}\label{sec:lip}
Coherent risk measures are subdifferentiable, see, e.g., \cite[Section 3]{ruszczynski2006optimization}. Since bounded subgradients imply Lipschitz continuity, this motivates us to demonstrate that the parametric value function is Lipschitz continuous in the state and parameter. This result allows us to establish explicit perturbation bounds for the value functions of risk-sensitive MDPs. 

For any two metric spaces $\CAL Y\times\CAL U$ with metrics $d_{\CAL Y}$ and $d_{\CAL U}$, we define the metric on $\CAL Y\times\CAL U$ to be
\begin{align*}
    d_{\CAL Y\times\CAL U}((y,u),(y',u')):=d_{\CAL Y}(y,y')+d_{\CAL U}(u,u').
\end{align*}
Let $2^{\CAL U}$ denote the set of all compact subsets of $\CAL U$. We endow $2^{\CAL U}$ with the Hausdorff metric
\begin{align*}
    d_H(U,U') := \max \left\{ \sup_{u\in U} d_{\CAL U}(u, U'), \sup_{u'\in U'} d_{\CAL U}(u', U)\right\},
\end{align*}
for all $U,U' \in 2^{\CAL U}$, where 
\begin{align*}
    d_{\CAL U}(u',U):=\inf_{u\in U}d_{\CAL U}(u',u).
\end{align*}

We recall from \cite[Definition (ii), p.5]{hinderer2005lipschitz}: a mapping $f:\CAL X\times\CAL Y\to\CAL Z$ is uniformly Lipschitz continuous on $\CAL Y$ if
\begin{align*}
    \sup_{y\in\CAL Y}\sup_{x\neq x'}\frac{d_{\CAL Z}(f(x,y), f(x',y))}{d_{\CAL X}(x,x')}<\infty.
\end{align*}

By \cite[Corollary 3.1]{ruszczynski2006optimization}, a coherent risk measure $\rho_{\theta,t}$ is continuous and subdifferentiable on the interior of its domain. Thus, by making additional boundedness assumptions on $\rho_{\theta, t}$, we can preserve Lipschitz continuity over the entirety of $\dom{\rho_{\theta,t}}$ by \cite[Lemma 2.1]{inoue2003worst}. In this case, we let $L_{\rho_{\theta,t}}<\infty$ be the Lipschitz coefficient of $\rho_{\theta,t}$ for every $\theta\in\Theta$, and we assume $L_{\Theta,t}:=\sup_{\theta\in\Theta}\{L_{\rho_{\theta,t}}\}<\infty$.

Let us define $\CAL M_{W_1}(\CAL S)$ as the set of measures over $\CAL S$ endowed with the Wasserstein metric, denoted by $W_1$, (which makes it a complete separable metric space). For a risk measure $\rho_{\theta,t}$, define $\Psi_t:\CAL S\times\CAL A\times\Theta\rightrightarrows \CAL M_{W_1}(\CAL S)$ as 
\begin{align}\label{eqn:Psi}
    \Psi_t(s,a,\theta) = \left\{\psi\in\CAL M_{W_1}(\CAL S): \psi(ds') = \phi(s')q(ds'|s,a,\theta), \phi\in\Phi_t(s,q(s,a,\theta),\theta)\right\}.
\end{align}
We can define the usual Haussdorff metric on the compact subsets of $\CAL M_{W_1}(\CAL S)$.

\begin{assumption}[Lipschitz MDP]\label{asm:lip}
The following statements hold:
\begin{enumerate}[label=(\roman*)]
    \item \label{asm:lip:A} $\Gamma$ is compact-valued and there exists $L_D \geq 0$, such that, for all $s, s' \in \CAL S$, we have
    \begin{align*}
        d_H(\Gamma(s,\theta), \Gamma(s',\theta')) \leq L_D \Big(d_{\CAL S}(s, s')+d_\Theta(\theta,\theta')\Big).
    \end{align*}
    \item The correspondence $\Psi$ is compact-valued and Lipschitz continuous with Lipschitz coefficient $L_{\rho,t}$:\label{asm:lip:rho}
    \begin{align*}
        d_W(\Psi(s,a,\theta), \Psi(s',a',\theta')) \leq L_{\rho,t} \Big(d_{\CAL S}(s, s')+d_{\CAL A}(a,a') + d_\Theta(\theta,\theta') \Big).
    \end{align*}
    \item The cost function $c_t$ is $L_{c_t}$-Lipschitz continuous on $\CAL S\times\CAL A\times\Theta$.\label{asm:lip:c}
    \item For infinite-horizon risk-sensitive MDP, $L_{c_t}\equiv L_c$ and $L_{\rho_t}\equiv L_\rho$. Further, $L_D$ in \ref{asm:lip:A} satisfy $(1+L_\rho)(1+L_D) < 1/\bar\gamma$.\label{asm:lip:gamma}
\end{enumerate}
\end{assumption}


\begin{theorem}\label{thm:mdp_lip}
Suppose \cref{asm:lip} holds. Also suppose the conditional risk measures $\{\rho_{\theta,t}\}_{t\in\Na}$ are Markov, coherent, and such that $\Phi_t$ is jointly continuous on $\CAL S\times\CAL M(\CAL S)\times\Theta$. 
Then
\begin{enumerate}[label=(\roman*)]
    \item $\min_{\pi\in\Pi}J_T(\cdot,\cdot;\pi)$ is $L_{v_0}$-Lipschitz, where $L_{v_t}$ is recursively defined as
    \begin{align*}
        L_{v_T}=&L_{c_T}, \\
        L_{v_t}=&L_{c_t} (1+L_D)+L_{v_{t+1}} (1+L_{\rho_{\theta,t}}) (1+L_D), \text{ for } t=0, \ldots, T-1.
    \end{align*}
    \item Suppose in addition $\{\rho_{\theta,t}\}_{t\in\Na}$ are stationary, then $\min_{\pi\in\Pi}J_\infty(\cdot,\cdot;\pi)$ is $L_\infty$-Lipschitz where
    \begin{align*}
        L_\infty = \frac{L_c(1+L_D)}{1-\bar\gamma (1+L_\rho)(1+L_D) }.
    \end{align*}
\end{enumerate}
\end{theorem}

\subsection{Discussion}

\cref{thm:mdp_cts} and \cref{thm:mono_cts} resolve the questions Q1 and Q2 under different hypotheses on the risk sensitive MDPs. Indeed, the continuity of $\min_{\pi\in\Pi}J_T(\cdot, \cdot)$ and $\min_{\pi\in\Pi}J_\infty(\cdot, \cdot)$ follows by assuming that: (i) the one-step risk measures are Markov and coherent (and also stationary for the infinite-horizon case) with continuous risk envelopes; and (ii) the cost function, the transition kernel, the admissible action set, and the discount factor are all continuous (as given in \cref{asm:cts_v} and \cref{asm:mono} with \cref{asm:scts_v}).

Under further assumptions on the MDP and the risk measure, \cref{thm:mdp_lip}  gives explicit bounds on the difference in value functions obtained from the nominal parameter and perturbed parameter, since the value functions are Lipschitz in the parameter.

\section{Proofs of Main Results}
\label{sec:proof}

We prove \cref{thm:mdp_cts}, \cref{thm:mono_cts}, and \cref{thm:mdp_lip} in this section. 

\subsection{Proof of Theorem \ref{thm:mdp_cts}}\label{pf:mdp_cts}

These claims are proven with the help of \cite[Theorem 1]{dutta1994parametric}, \cite[Theorem 2]{ruszczynski2010risk}, and \cite[Theorem 4]{ruszczynski2010risk}. We first give the proof for $\FRK P_\infty$, and then elaborate on the variations needed for $\FRK P_T$.

{\bf Case of $\FRK P_\infty$: }The proof for $\FRK P_\infty$ consists of the following steps:
\begin{enumerate}
    \item We apply \cite[Theorem 4]{ruszczynski2010risk} to show that $\FRK P_\infty$ has a DP decomposition.
    \item We next show that the risk-sensitive Bellman operator is a contraction in $\CAL C_b(\CAL S\times\Theta)$. 
    \item Then, the conclusion follows by Berge's Maximum Theorem.  
\end{enumerate}

We start by verifying conditions (i)-(v)\footnote{The joint continuity of $q$, $\Phi$, $c$ and $\CAL A$ are further required by \cite[p.604]{ruszczynski2014erratum}.} required by \cite[Theorem 4]{ruszczynski2010risk} for $\FRK P_\infty$. For every $\theta\in\Theta$:

\begin{enumerate}
    \item By \cref{asm:cts_v} \ref{asm:cts_v:q}, $q(\cdot,\cdot,\theta)$ is setwise continuous, which yields condition (i).
    \item The $\{\rho_{\theta,t}\}_{t\in\Na}$ in \cref{def:markov} are stationary Markov risk measures. Then, the fact that $\Phi$ is continuous yields condition (ii).
    \item By \cref{asm:cts_v} \ref{asm:cts_v:c}, boundedness and continuity of $c$ yields conditions (iii) and (iv).
    \item \cref{asm:cts_v} \ref{asm:cts_v:A} yields condition (v).
\end{enumerate}

Then, by \cite[Theorem 4]{ruszczynski2010risk}, $\FRK P_\infty$ can be solved by computing the optimal risk-sensitive value function $v(\cdot,\theta)$ which satisfies:
\begin {align}\label{eq:v_sigma}
    v(s,\theta)=&\min_{a\in\Gamma(s,\theta)} c(s,a,\theta) + \gamma(\theta)\sigma(v(\cdot,\theta), s, q(s,a,\theta),\theta),\, \forall s \in \CAL S,
\end{align}
where $\sigma$ is defined in \cref{eq:rho_sigma}. Now define the mapping $\hat{\CAL T}:\CAL C_b(\CAL S\times\Theta)\to\CAL C_b(\CAL S\times\Theta)$ by
\begin{align}\label{eq:T_hat}
    \hat{\CAL T}(\hat v)(s,\theta) :=& \min_{a\in\Gamma(s,\theta)} c(s,a,\theta)+ \gamma(\theta)\sigma(\hat v(\cdot,\theta),s,q(s,a,\theta),\theta),\, \forall (s, \theta) \in \CAL S \times \Theta.
\end{align}
We show that $v$ in \cref{eq:v_sigma} is the fixed point of $\hat{\CAL T}$ (i.e., $v=\hat{\CAL T}(v)$), which requires the following auxiliary result.
\begin{lemma}\label{lem:sigma_cts}
For any $\hat v\in\CAL C_b(\CAL S\times\Theta)$, the mapping:
\begin{align}\label{eq:sigma_qv}
    \sigma_{q,\hat v}:(s,a,\theta)\mapsto \sigma(\hat v(\cdot,\theta),s,q(s,a,\theta),\theta),
\end{align}
is jointly continuous and bounded on $\CAL S\times\CAL A\times\Theta$. 
\end{lemma}
\begin{proof}
See \cref{pf:sigma_cts}.
\end{proof}

By \cref{lem:sigma_cts},
together with the continuity of $c$, $\CAL A$, and $\gamma$ from \cref{asm:cts_v} \ref{asm:cts_v:c}, \ref{asm:cts_v:A}, and \ref{asm:cts_v:gamma}, we establish that $\hat{\CAL T}(\hat v)\in\CAL C_b(\CAL S\times\Theta)$ by applying Berge's Maximum Theorem to the RHS of \cref{eq:T_hat}. Then, $\hat{\CAL T}$ is a contraction mapping by the following result.

\begin{lemma}\label{lem:T_ctrct}
The mapping $\hat{\CAL T}:\CAL C_b(\CAL S\times\Theta)\to\CAL C_b(\CAL S\times\Theta)$ is a $\bar \gamma$-contraction in the supremum norm.
\end{lemma}
\begin{proof}
See \cref{pf:T_ctrct}.
\end{proof}

Now, by \cref{lem:T_ctrct} we see that $v$ defined in \cref{eq:v_sigma} is the unique fixed point of $\hat{\CAL T}$. Thus, the fixed point satisfies $v\in\CAL C_b(\CAL S\times\Theta)$, as $\CAL C_b(\CAL S\times\Theta)$ is complete under the supremum norm, which proves the first part of \cref{thm:mdp_cts} since $v(\cdot,\cdot) = \min_{\pi\in\Pi}J_\infty(\cdot,\cdot;\pi)$.

For the second part, by \cite[Theorem 4]{ruszczynski2010risk} the optimal policy $\pi^*$ exists, and for each $\theta\in\Theta$ it satisfies:
\begin{align}\label{eq:pi}
    \pi^*(s,\theta)\in &\arg\min_{a\in\Gamma(s,\theta)} c(s,a,\theta)+\gamma(\theta) \sigma( v(\cdot,\theta), s, q(s,a,\theta),\theta),\, \forall s \in \CAL S.
\end{align}
Again, \cref{lem:sigma_cts} along with \cref{asm:cts_v} \ref{asm:cts_v:c} and \ref{asm:cts_v:A} imply that $\pi^*(\cdot,\cdot)$ is lower semicontinuous on $\CAL S\times\Theta$ by Berge's Maximum Theorem. This completes the proof of continuity for $\FRK P_\infty$. 

{\bf Case of $\FRK P_T$: } The proof for $\FRK P_T$ is by mathematical induction. Note that \cref{asm:cts_v} and \cref{def:markov} give the conditions required by \cite[Theorem 2]{ruszczynski2010risk}, and thus $\min_{\pi\in\Pi} J(s_0,\theta;\pi)$ can be solved by the DP decomposition:
\begin{align*}
    v_T(s,\theta)=&c_T(s,\theta),\\
    v_t(s,\theta)=&\min_{a\in\Gamma(s,\theta)}c_t(s,a,\theta) + \sigma_t(v_{t+1}(\cdot,\theta), s, q_t(s,a,\theta),\theta),\, t = 0, \ldots, T-1,
\end{align*}
which gives the initial value function $v_0$. Starting from $t=T$, the terminal cost $v_T=c_T$ is jointly continuous and bounded by \cref{asm:cts_v} \ref{asm:cts_v:c}. For the induction step for time $t=0, \ldots, T-1$, suppose that $v_{t+1}\in\CAL C_b(\CAL S\times\Theta)$. Then, by \cref{lem:sigma_cts}, the mapping $\sigma_{q_t,v_{t+1}}$ is continuous and bounded. Thus, \cref{asm:cts_v} implies that $v_t\in\CAL C_b(\CAL S\times\Theta)$ and so $\pi_t^*$ is lower semicontinuous by Berge's Maximum Theorem, which completes the induction step. Taking $v_0(s_0,\theta) = \min_{\pi\in\Pi} J_T(s_0,\theta;\pi)$, we establish continuity for $\FRK P_T$.

\subsection{Proof of Theorem \ref{thm:mono_cts}}\label{pf:thm:mono_cts}
This proof is similar to \cite[Theorem 3]{dutta1994parametric}, except we work with the risk transition mapping $\sigma$ instead of the usual expectation of the value function. 


{\bf Case of $\FRK P_\infty$: }For every $\theta\in\Theta$, \cref{asm:scts_v} and \cref{def:markov} yield the conditions required by \cite[Theorem 4]{ruszczynski2010risk}, which then implies that the optimal value function $v$ solves the DP decomposition \cref{eq:v_sigma}.

Let $\CAL C_{b}^\uparrow(\CAL S\times\Theta)$ be the set of bounded and continuous functions on $\CAL S\times\Theta$ that are monotonically increasing on $\CAL S$. We will show that for any $\tilde v\in\CAL C_b^\uparrow(\CAL S\times\Theta)$, the mapping $\tilde{\CAL T}:\CAL C_b^\uparrow(\CAL S\times\Theta)\to\CAL C_b^\uparrow(\CAL S\times\Theta)$, defined by
\begin{align}\label{eq:T_tilde}
    \tilde{\CAL T}(\tilde v)(s,\theta) :=& \min_{a\in\Gamma(s,\theta)} c(s,a,\theta)+\gamma(\theta)\sigma(\tilde v(\cdot,\theta),s,q(s,a,\theta),\theta),\, \forall (s, \theta) \in \CAL S \times \Theta,
\end{align}
has the value function $v$ defined in \cref{eq:v_sigma} as its unique fixed point.

\begin{lemma}\label{lem:sigma_mb}
The mapping:
\begin{align*}
    \sigma_{q,\tilde v}:(s,a,\theta)\mapsto \sigma(\tilde v(\cdot,\theta),s,q(s,a,\theta),\theta),
\end{align*}
is continuous on $\CAL S\times\CAL A$ for every $\theta\in\Theta$ and continuous on $\CAL A\times\Theta$ for every $s\in\CAL S$.
\end{lemma}
\begin{proof}
The proof is similar to the proof of \cref{lem:sigma_cts}. By fixing $\theta\in\Theta$, we can prove that $\sigma_{q,\tilde v}$ is jointly continuous on $\CAL S\times\CAL A$, and by fixing $s\in\CAL S$, we can prove that $\sigma_{q,\tilde v}$ is jointly continuous on $\CAL A\times\Theta$. We omit the details here for brevity.
\end{proof}

Now, for any $t\in\Na$, $(a,\theta)\in\CAL A\times\Theta$, and $s_t,s_t'\in\CAL S$, let $S_{t+1}\sim q(s_t,a,\theta)$ and $S_{t+1}'\sim q(s_t',a,\theta)$. Then $s_t\leq s_t'$ implies $S_{t+1}\preceq S_{t+1}'$ according to \cref{asm:mono} \ref{asm:mono:q}. Thus, \cref{eq:rho_sigma} yields
\begin{align*}
    \sigma(\tilde v(\cdot,\theta),s_t,q(s_t,a,\theta),\theta)=&\rho_{\theta,t}(\tilde v(S_{t+1},\theta)) \nonumber\\
    \stackrel{(a)}{\leq}& \rho_{\theta,t}(\tilde v(S_{t+1}',\theta)) = \sigma(\tilde v(\cdot,\theta),s_t',q(s_t',a,\theta),\theta),
\end{align*}
where inequality (a) holds because $\tilde v\in\CAL C_b^\uparrow(\CAL S\times\Theta)$ is monotone in $s$, and the coherent risk measure $\rho_{\theta,t}$ preserves the first stochastic order by \cite[Lemma 5.1]{ruszczynski2006optimization}. This directly implies that the mapping $\sigma_{q,\tilde v}$ is also monotone on $\CAL S$. Combined with \cref{lem:sigma_mb} we have that $\sigma_{q,\tilde v}$ is jointly continuous on $\CAL S\times\CAL A\times\Theta$ as a result of \cite[Lemma 2]{dutta1994parametric}. Then, applying Berge's Maximum Theorem to \cref{eq:T_tilde} we obtain $\tilde{\CAL T}(\tilde v)\in\CAL C_b(\CAL S\times\Theta)$. 

Furthermore, for any $s,s'\in\CAL S$ with $s\leq s'$, we have
\begin{align*}
    &\min_{a\in\Gamma(s,\theta)} c(s,a,\theta)+\gamma(\theta)\sigma(\tilde v(\cdot,\theta),s,q(s,a,\theta),\theta)\\
    & \qquad \leq \min_{a\in\Gamma(s',\theta)} c(s,a,\theta)+\gamma(\theta)\sigma(\tilde v(\cdot,\theta),s,q(s,a,\theta),\theta)\\
    &\qquad \leq \min_{a\in\Gamma(s',\theta)} c(s',a,\theta)+\gamma(\theta)\sigma(\tilde v(\cdot,\theta),s',q(s',a,\theta),\theta),
\end{align*}
by \cref{asm:mono} and the monotonicity of $\sigma_{q,\tilde v}$. It follows that $\tilde{\CAL T}(\tilde v)$ is also monotone on $\CAL S$, and so $\tilde{\CAL T}(\tilde v)\in\CAL C_b^\uparrow(\CAL S\times\Theta)$. Therefore, $\tilde{\CAL T}$ is a contraction on $\CAL C_b^\uparrow(\CAL S\times\Theta)$ with $v$ as its fixed point by reasoning similarly to \cref{lem:T_ctrct}. 
Consequently, using Berge's Maximum Theorem again we conclude that $\pi^*$ is lower semicontinuous. 

{\bf Case of $\FRK P_T$: }The finite-horizon MDP can be proven analogously with mathematical induction, we omit the details here for brevity.

\subsection{Proof of Theorem \ref{thm:mdp_lip}}\label{pf:mdp_lip}
We apply the results of \cite{hinderer2005lipschitz} to establish Lipschitz continuity of the value functions for both $\FRK P_\infty$ and $\FRK P_T$.

{\bf Case of $\FRK P_\infty$: }\cref{asm:lip} implies \cref{asm:cts_v}, which in turn implies $v(s_0,\theta)=\min_{\pi\in\Pi}J_\infty(s_0,\theta;\pi)$ can be computed by \cref{eq:v_sigma}. We recall the mapping $\hat{\CAL T}$ defined in \cref{eq:T_hat}, and we restrict its domain to $\CAL C_{L_0}(\CAL S\times\Theta)$ for some $L_0<\infty$. We can then apply the results of \cite{hinderer2005lipschitz} to show that its fixed point $v=\hat{\CAL T}(v)$ is Lipschitz continuous. This argument requires the following auxiliary lemmas.

\begin{lemma}\label{lem:rho_lip}
The mapping $(s,\theta)\mapsto\rho_{\theta,t}(\bar v(s,\theta))$ is $L_0 L_\rho$-Lipschitz continuous on $\CAL S\times\Theta$ for any $\bar v\in\CAL C_{L_0}(\CAL S\times\Theta)$.
\end{lemma}
\begin{proof}
See \cref{pf:rho_lip}.
\end{proof}

\begin{lemma}\label{lem:sigma_lip}
The mapping $\sigma_{q,\bar v}$ defined in \cref{eq:sigma_qv} is $L_0 (1+L_\rho)$-Lipschitz continuous for any $\bar v\in\CAL C_{L_0}(\CAL S\times\Theta)$. 
\end{lemma}
\begin{proof}
See \cref{pf:sigma_lip}.
\end{proof}

As a consequence of \cref{lem:sigma_lip} with \cref{asm:lip} parts \ref{asm:lip:A}, \ref{asm:lip:c}, and \ref{asm:lip:gamma}, we establish that the fixed point $v=\hat{\CAL T}(v)$ is $L_\infty$-Lipschitz continuous by \cite[Theorem 4.1]{hinderer2005lipschitz}, where
\begin{align*}
    L_\infty = \frac{L_c(1+L_D)}{1-\bar\gamma (1+L_\rho)(1+L_D) }.
\end{align*}
This completes the first part of the proof.

{\bf Case of $\FRK P_T$: }Lipschitz continuity for the value functions of the finite-horizon problem is again proven by mathematical induction. First note that the terminal cost $v_T=c_T$ is $L_{c_T}$-Lipschitz continuous. For the induction step for time $t=0,\ldots,T-1$, assume $v_{t+1}$ is $L_{v_{t+1}}$-Lipschitz continuous, then \cref{lem:sigma_lip} yields that the mapping $\sigma_{q_t,v_{t+1}}$ is $L_{v_{t+1}}(1+L_{\rho,t})$-Lipschitz continuous. Next, by \cite[Lemma 3.2]{hinderer2005lipschitz}, $v_t$ is $L_{v_t}$-Lipschitz continuous where $L_{v_t}=L_{c_t}(1+L_D)+L_{v_{t+1}}(1+L_{\rho,t})(1+L_D)<\infty$, which completes the induction step. By picking $v_0(\cdot,\cdot) = \min_{\pi\in\Pi}J_T(\cdot,\cdot;\pi)$, we prove the result for $\FRK P_t$.

\section{Continuity of Risk Measures and Risk Transition Mappings}\label{sec:example}
Our main results assume the joint continuity of the risk envelope $\Phi$ on $\CAL S\times\CAL M(\CAL S)\times\Theta$. Under this condition, through an application of Berge's Maximum Theorem, we are able to demonstrate that the risk measure is continuous in its arguments. This technique compels us to ask two questions:

1. Are there general parametric risk measures that are continuous over closed subsets of $\CAL X\times\Theta$?

2. Are there sufficient conditions under which risk transition mappings are continuous?

We devote this section to answering these two questions. First, we focus on deriving classes of parametric risk measures that are continuous under certain assumptions. Then, we identify a sufficient condition on the risk envelope under which it is both upper and lower hemicontinuous (and therefore, a continuous correspondence). A simple application of Berge's Maximum Theorem then yields the desired continuity of the risk transition mapping.

\subsection{Continuity of Risk Measures}



In this section, we identify some examples of risk measures that are continuous on $\CAL X\times\Theta$. We first recall the following useful result regarding convergence in $\CAL L_1$.

\begin{theorem}[\cite{bogachev2007measure}, Theorem 4.5.4]\label{thm:L1conv}
Let $\BBM P$ be a probability measure. Suppose that $f$ is a $\BBM P-$measurable function and $\{f_n\}_{n\in\Na}$ is a sequence of $\BBM P$-integrable functions. Then the following assertions are equivalent:
\begin{enumerate}
    \item The sequence $\{f_n\}_{n\in\Na}$ converges to $f$ in measure and is uniformly integrable.
    \item  The function $f$ is integrable and the sequence $\{f_n\}_{n\in\Na}$ converges to $f$ in $\CAL L^1$.
\end{enumerate}
\end{theorem}

Note that if $1\leq q<p<\infty$, then $X_n\cp X$ implies $X_n\stackrel{\CAL L_q}{\to}X$ since
\begin{align*}
    \ex{\abs{X_n-X}^{q}}^{\frac{1}{q}}=\left(\ex{\abs{X_n-X}^q}^{\frac{p}{q}}\right)^{\frac{1}{p}}\stackrel{(a)}{\leq} \ex{\abs{X_n-X}^{p}}^{\frac{1}{p}},
\end{align*}
where inequality (a) holds by Jensen's inequality since $f(x)=x^{\frac{p}{q}}$ is convex in $x$. Furthermore, $\{X_n\}_{n\in\Na}\subset\CAL X$ implies $\abs{X_n}^p$ is $\BBM P$-integrable, and so $X_n$ is $\BBM P$-integrable. That is, the conditions required by \cref{thm:L1conv} are fulfilled if we assume $X_n\cp X$.

In all of the examples below, $X$ is a real-valued random variable. In addition, we introduce a function $\lambda:\Theta\to\Re_+$ that expresses the DM's degree of risk-sensitivity (higher values of $\lambda(\theta)$ mean the DM is more risk-sensitive).

\begin{example}[Worst-Loss Risk Measure]\label{exm:inf}
\normalfont 
Consider a risk measure:
\begin{align*}
    \rho_\theta(X):=\ex{X}+\lambda(\theta)\text{ess}\inf X,
\end{align*}
where $\text{ess}\inf$ is the essential infimum of $X$, i.e.,
\begin{align*}
    \text{ess}\inf X=\sup\left\{x\in\Re:\pr{X<x}=0\right\}.
\end{align*}
This $\rho_\theta$ satisfies \cref{def:coherent}. 

Now suppose $\Omega$ is compact and $X$ and the sequence $\{X_n\}_{n\in\Na}$ are continuous on $\Omega$, so $\norm{X}_\infty<\infty$ and we can replace $\text{ess}\inf$ with $\inf$. It follows that
\begin{align*}
    \rho_\theta(X)=\ex{X}+\lambda(\theta)\inf_{\omega\in\Omega} X(\omega),
\end{align*}
is continuous on $\CAL C_b(\Omega)\times\Theta$. This claim is based on the following observations. First, $X_n\cp X$ for $p\geq 1$ implies $X_n\to X$ in $\BBM P$ and that $\{X_n\}_{n\in\Na}$ is $\BBM P$-uniformly integrable by \cref{thm:L1conv}, which directly implies $\ex{X_n}\to\ex{X}$. 
Next, since $\Omega$ is compact and $X$ is continuous, by Berge's Maximum Theorem we have $\inf_{\omega\in\Omega} X_n(\omega)\to \inf_{\omega\in\Omega} X(\omega)$. Then, $\lambda(\theta_n)\inf X_n(\omega)\to \lambda(\theta)\inf X(\omega)$ since $\lambda$ is also continuous. It follows that $\rho_\theta(X)$ is jointly continuous.
\end{example}

\begin{example}[Mean-Deviation of Order $p$]\label{exm:md}
\normalfont
For $p\in[1,\infty)$, define
\begin{align*}
    \rho_\theta(X):=\ex{X}+\lambda(\theta)\norm{X-\ex{X}}_p.
\end{align*}
For $p=2$, this is the standard mean-deviation model introduced in \cite{Markowitz1952portfolio}. In this case, $\rho_\theta$ meets all but the monotonicity requirement of \cref{def:coherent} for $p>1$. However, for $p=1$, if $\BBM P$ is non-atomic, then $\rho_\theta$ is coherent if and only if $\lambda(\theta)\in[0,1/2]$ by \cite[Example 6.19]{shapiro2021lectures}.

We establish continuity for $p=1$ and $\lambda(\theta_n)\in[0,1/2]$. Indeed, by the triangle inequality:
\begin{align}\label{eq:jc:md}
    &0\leq\lim_{n\to\infty}\left(\norm{X_n-\ex{X_n}}_1 -\norm{X-\ex{X}}_1\right)\nonumber\\
    &\qquad\leq \lim_{n\to\infty}\left(\norm{X_n-X}_1 +\abs{\ex{X_n}-\ex{X}}\right) \stackrel{(a)}{=}0,
\end{align}
where equality (a) holds since: (i) $\lim_{n\to\infty}\norm{X_n-X}_1=0$ by $X_n\stackrel{\CAL L_1}{\to} X$, and (ii) $\lim_{n\to\infty}\abs{\ex{X_n}-\ex{X}}=0$ as we have shown in \cref{exm:inf}. Then it follows that
\begin{align*}
    \lambda(\theta_n)\norm{X_n-\ex{X_n}}_1\to\lambda(\theta)\norm{X-\ex{X}}_1,
\end{align*}
which implies $\rho_\theta(X)$ is jointly continuous.
\end{example}

\begin{example}[Mean-Upper-Semideviation of Order $p$]
\normalfont
Recall
\begin{align*}
    \rho_\theta(X):=\ex{X}+ \lambda(\theta)\norm{(X-\ex{X})_+}_p,
\end{align*}
is the mean-upper-semideviation, which is a coherent risk measure.
In contrast to the mean-deviation, $\rho_\theta$ is monotonic if $\lambda(\theta)\in[0,1]$ and $\BBM P$ is non-atomic \cite{ogryczak1999stochastic}.

We show that if $\lambda(\theta)\in[0,1]$, then \begin{align*}
    \rho_\theta(X):=\ex{X}+ \lambda(\theta)\norm{(X-\ex{X})_+}_p,
\end{align*}
is continuous. Indeed, since $\abs{(X_n-\ex{X_n})_+}^p\leq\abs{X_n-\ex{X_n}}^p$ for all $n\in\Na$ and $p\in[1,\infty)$, the Dominated Convergence Theorem implies
\begin{align*}
    \norm{(X_n-\ex{X_n})_+}_p\to\norm{(X - \ex{X})_+}_p, \text{ as } n\to\infty,
\end{align*}
where the convergence $\norm{X_n-\ex{X_n}}_p\to\norm{X - \ex{X}}_p$ is based on a similar argument as \cref{eq:jc:md}. This yields the joint continuity of $\rho_\theta(X)$.
\end{example}

\begin{example}[Certainty Equivalent]
\normalfont
Let $\FRK U_\theta:\Re\to\Re$ be a utility function which is continuous and monotonically increasing (and thus $\FRK U_\theta^{-1}$ exists), for all $\theta\in\Theta$. The corresponding certainty equivalent is $\rho_{\theta}(X) = \FRK U_\theta^{-1}(\ex{\FRK U_\theta(X)})$. For general $\FRK U_\theta$, the certainty equivalent is not coherent as it may fail to satisfy positive homogeneity and convexity. However, for the exponential utility function $\FRK U_\theta(x)=\exp(\lambda(\theta) x)/\lambda(\theta)$, the homogenization procedure \cite{shapiro2021lectures} produces the following coherent risk measure:
\begin{align}\label{eq:rho_util}
    \rho_\theta(X)=\inf_{\tau>0} \tau\, \FRK U_\theta^{-1} \left(\ex{\FRK U_\theta\left(\frac{X}{\tau}\right)}\right) \text{ is coherent}.
\end{align}

Suppose $\FRK U_\theta(X)$ is jointly continuous on $\CAL V_{\CAL S\times\Theta}\times\Theta$. If we further assume that the homogenization given in \cref{eq:rho_util} satisfies
\begin{align*}
    \lim_{\tau\downarrow 0}\tau\FRK U_\theta^{-1} \ex{\FRK U_\theta(X/\tau)} =\infty,
\end{align*}
then $X \to \tau\FRK U_\theta^{-1} \ex{\FRK U_\theta(X/\tau)}$ is $\BBM K\BBM N$-inf-compact on the graph $\{(\tau,X,\theta)\in\Re\times\CAL V_{\CAL S\times\Theta}\times\Theta: \tau> 0\}$ by \cite[Definition 1.3]{feinberg2014berges}. Then, the generalization of Berge's Maximum Theorem \cite[Theorem 1.4]{feinberg2014berges} yields the joint continuity of $\rho_\theta(X)=\inf_{\tau>0}\tau\FRK U_\theta^{-1} \ex{\FRK U_\theta(X/\tau)}$ on $\CAL V_{\CAL S\times\Theta}\times\Theta$.

\end{example}

\begin{example}[Conditional Value-at-Risk]\label{exm:cvar_cts}
\normalfont
The Value-at-Risk (VaR) at \\ level $u\in[0,1)$ is defined as
\begin{align}\label{eq:VaR}
    \text{VaR}_u(X):=F_X^{-1}(u)=\inf\left\{x:\pr{X\leq x}\geq u\right\},
\end{align}
where $F_X^{-1}$ is also referred to as the inverse CDF or quantile function. It is well known that VaR does not satisfy the sub-additivity property (which is implied by convexity and positive homogeneity). A common alternative to VaR is the CVaR (at level $\lambda(\theta)\in[0,1)$), which is defined as:
\begin{align*}
    \text{CVaR}_{\lambda(\theta)}(X):=\frac{1}{1-\lambda(\theta)}\int_{\lambda(\theta)}^1 \text{VaR}_u(X) du,
\end{align*}
and we let $\rho_\theta(X)=\text{CVaR}_{\lambda(\theta)}(X)$.

Suppose the CDF $F_{X_n}$ of $X_n$ is continuous for all $n \in \Na$, and all $X_n:\Omega\to[0,\infty)$. We will show that $\text{CVaR}_{\lambda(\theta)}(X)$ is jointly continuous on $\CAL V_{\CAL S\times\Theta}\times\Theta$ by the Dominated Convergence Theorem in \cref{pf:cvar_cts}. 
\end{example}

\subsection{Continuity of Risk Transition Mapping}
We now address the question of establishing the continuity of the risk transition mapping defined in \eqref{eqn:Phirtm}. Towards this end, define $\CAL H$ to be the set of sequentially continuous functions:
\begin{align*}
    \CAL H &= \Big\{h:\CAL S\times \CAL L_p^*(\CAL S,\CAL B(\CAL S),\mathbb{Q}) \times \Theta\to \CAL L_p^*(\CAL S,\CAL B(\CAL S),\mathbb{Q}): h \text{ is sequentially continuous}\Big\}.
\end{align*}
For a transition kernel $q(s,a,\theta)\in\CAL M(\CAL S)$, we identify it by its Radon-Nikodym derivative $m\in\CAL L_p^*(\CAL S,\CAL B(\CAL S),\BBM Q)$ with respect to $\BBM Q$. We use $M$ to denote this mapping, that is, $m(s') = M(q(s,a,\theta))(s')\in \CAL L_p^*(\CAL S,\CAL B(\CAL S),\mathbb{Q})$. 
We make the following assumption to proceed.

\begin{assumption}\label{asm:phi_cts}
Let $\CAL H' \subset\CAL H$ be any subset. There exists $\bar\phi\in \CAL L_p^*(\CAL S,\CAL B(\CAL S),\mathbb{Q})$ satisfying $\int \bar\phi(s') \mathbb{Q}(ds')<\infty$ such that $\Phi(s,q(s,a,\theta),\theta)$ is given by
\begin{align}\label{eq:Phi_smth}
    \Phi(s,q(s,a,\theta),\theta) &= \Big\{\phi\in \CAL L_p^*(\CAL S,\CAL B(\CAL S),\mathbb{Q}): \phi = h(s,m,\theta)m, h\in\CAL H' \subset\CAL H,\nonumber\\
    & \qquad m = M(q(s,a,\theta)), \; 0\leq \phi(s')\leq \bar \phi(s')\Big\}.
\end{align}
\end{assumption}

\begin{lemma}\label{lem:phi_cts}
Suppose that the Assumption \ref{asm:phi_cts} holds. Then, the correspondence $\Phi:\CAL S\times\CAL L_p^*(\CAL S,\CAL B(\CAL S),\mathbb{Q})\times\Theta \rightrightarrows \CAL L_p^*(\CAL S,\CAL B(\CAL S),\BBM Q)$ defined in \eqref{eqn:Phirtm} is a continuous correspondence.
\end{lemma}
\begin{proof}
See \cref{pf:phi_cts} for the proof. 
\end{proof}
One can now simply apply Berge's Maximum Theorem to show that the risk transition mapping $\sigma_t(\cdot, s,q_t(s,a,\theta),\theta)$ is continuous in $v(\cdot,\theta)$. We note here that in \cite{ruszczynski2014erratum}, the author provides an example of a continuous risk-transition mapping in Example 1. Lemma \ref{lem:phi_cts} generalizes that result and \cite[Example 1]{ruszczynski2014erratum} is a special case of Lemma \ref{lem:phi_cts}.

\section{Conclusion}\label{sec:conclusion}
In this paper, we consider risk-sensitive MDPs based on nested Markov risk measures, as elucidated in \cite{ruszczynski2010risk}.
In our framework, both the system parameters and the DM's risk preferences are encoded through a model parameter that is subject to perturbation.
We examine sufficient conditions for the value functions to be continuous on the parameter space for both finite-horizon and infinite-horizon MDPs. Our first result requires the system model and the risk measure to be jointly continuous over the state, action, and parameter spaces. Then, we relax this assumption to only require separate continuity for monotone MDPs.
In this way, our results generalize the parametric continuity results for risk-neutral MDPs in \cite{dutta1994parametric} to the class of risk-sensitive MDPs.

\appendix

\section{Proof of Lemma \ref{lem:sigma_cts}}\label{pf:sigma_cts}
We prove the claim by applying Berge's Maximum Theorem \cite[Theorem 17.31]{Aliprantis2006} to $\sigma$ defined in \cref{eq:sigma}. For $\hat v\in\CAL C_b(\CAL S\times\Theta)$, we define:
\begin{align}\label{eq:inf_sigma}
    \sigma(\hat v(\cdot,\theta),s,q(s,a,\theta),\theta)=\sup_{\phi\in\Phi(s,q(s,a,\theta),\theta)}\inner{\hat v(\cdot,\theta),\phi}.
\end{align}
Consider the map
\begin{align}\label{eq:inner_vphi}
    (s,a,\theta,\phi)\mapsto \int \hat v(s',\theta)\phi(s')q(ds'|s,a,\theta),
\end{align}
where $\phi\in\CAL P_\CAL S$ and
\begin{align*}
    \CAL P_\CAL S:=\left\{\phi\in\CAL X^*:\int \phi(s)\mathbb{Q}(ds)=1,\;\phi\geq 0\right\}
\end{align*}
is endowed with the weak* topology.

For any sequence $\{\theta_n\}_{n\in\Na}\subset\Theta$ with $\theta_n\to\theta$, since $\hat v$ is jointly continuous on $\CAL S\times\Theta$, we have
\begin{align*}
     \left\{s'\in\CAL S:\lim_{n\to\infty}\hat v(s_n',\theta_n)\neq \hat v(s',\theta),\forall s_n'\to s'\right\}=\emptyset.
\end{align*}
Moreover, since $\phi_n\ws\phi$ and $q(s_n,a_n,\theta_n)\sw q(s,a,\theta)$, we conclude that 
\[\phi_n q(s_n,a_n,\theta_n)\ws \phi \, q(s,a,\theta).\]
Thus, for any $\hat v\in\CAL C_b(\CAL S\times\Theta)$,
\begin{align*}
    \int \hat v(s',\theta_n)\phi_n(s')q(ds'|s_n,a_n,\theta_n)\to \int \hat v(s',\theta)\phi(s')q(ds'|s,a,\theta),
\end{align*}
which establishes that the map \cref{eq:inner_vphi} is continuous by \cite[Theorem 5.5]{billingsley1968convergence}.

Next, by \cite[Proposition 6.2]{ruszczynski2006optimization} we have that $\Phi$ is weakly compact for every $(s,a,\theta)\in\CAL S\times\CAL A\times\Theta$. Then \cref{asm:cts_v} \ref{asm:cts_v:q} and the joint continuity of $\Phi$ imply that the mapping  $(s,a,\theta)\mapsto\Phi(s,q(s,a,\theta),\theta)$ is continuous. Thus, an application of Berge's Maximum Theorem \cite{Aliprantis2006} yields the continuity of $\sigma_{q,\hat v}$ on $\CAL S\times\CAL A\times\Theta$. 


\section{Proof of Lemma \ref{lem:T_ctrct}}\label{pf:T_ctrct}
Fix a policy $\pi\in\Pi$ and define the mapping $\hat{\CAL T}$ by
\begin{align*}
    \hat{\CAL T}_\pi(\hat v)(s,\theta)=c(s,\pi(s,\theta),\theta) +\gamma(\theta)\sigma(\hat v(\cdot,\theta),s,q(s,\pi(s,\theta),\theta),\theta).
\end{align*}
For any $\hat v_1,\hat v_2\in\CAL C_b(\CAL S\times\Theta)$ (endowed with the supremum norm) we have:
\begin{align*}
    &\norm{\hat{\CAL T}_\pi(\hat v_1)-\hat{\CAL T}_\pi(\hat v_2)}_\infty \\
    =& \Big\| c(s,\pi(s,\theta),\theta) +\gamma(\theta)\sigma(\hat v_1(\cdot,\theta),s,q(s,\pi(s,\theta),\theta),\theta)\\
    &\qquad - c(s,\pi(s,\theta),\theta) - \gamma(\theta)\sigma(\hat v_2(\cdot,\theta),s,q(s,\pi(s,\theta),\theta),\theta)\Big\|_\infty \\
    =&\sup_{(s,\theta)\in\CAL S\times\Theta}\Big|\gamma(\theta)\big(\sigma(\hat v_1(\cdot,\theta),s,q(s,\pi(s,\theta),\theta),\theta)-\sigma(\hat v_2(\cdot,\theta),s,q(s,\pi(s,\theta),\theta),\theta)\big)\Big|\\
    = &\sup_{(s,\theta)\in\CAL S\times\Theta}\abs{\gamma(\theta)\sup_{\phi\in\Phi(s,q(s,\pi(s,\theta),\theta),\theta)}\int (\hat v_1(s',\theta)-\hat v_2(s',\theta))\phi(s') \mathbb{Q}(ds')}\\
    \stackrel{(a)}{\leq}& \bar\gamma \sup_{(s,\theta)\in\CAL S\times\Theta}\abs{(\hat v_1(\cdot,\theta)-\hat v_2(\cdot,\theta))\sup_{\phi\in\Phi(s,q(s,\pi(s,\theta),\theta),\theta)}\int  \phi(s') \mathbb{Q}(ds')}\\
    \stackrel{(b)}{\leq}&\bar \gamma\norm{\hat v_1-\hat v_2}_\infty,
\end{align*}
where (a) holds due to the bound on the discount factor $\gamma(\theta)<\bar\gamma$ for every $\theta \in \Theta$, and (b) holds since $\Phi$ is a set of probability density functions. Then, $\hat{\CAL T_\pi}$ is a contraction for any $\pi\in\Pi$. Taking $\pi=\pi^*$ to be the optimal policy defined in \cref{eq:pi}, we conclude that $\hat{\CAL T}$ is also a contraction.

\section{Proof of Lemma \ref{lem:rho_lip}}\label{pf:rho_lip}
We first use \cite[Remark 2]{hinderer2005lipschitz} to show that $\rho_{\theta,t}(X)$ is jointly Lipschitz on $\CAL X\times\Theta$, then the result will hold by composition of Lipschitz continuous functions.

By \cite[Lemma 2.1]{inoue2003worst} and \cite[Corollary 3.1]{ruszczynski2006optimization}, we have that $\rho_{\theta,t}$ is $L_\rho$-Lipschitz continuous on the interior of $\dom{\rho_{\theta,t}}$. Then 
\begin{align*}
    \sup_{\theta\in\Theta}\sup_{X\neq X'}\frac{\norm{\rho_{\theta,t}(X)-\rho_{\theta,t}(X')}_\infty}{\norm{X-X'}_\infty}=\sup_{\theta\in\Theta}L_{\rho_{\theta,t}}\leq L_\rho<\infty,
\end{align*}
which establishes that $\rho_{\theta,t}(X)$ is uniformly Lipschitz continuous on $\Theta$. Then, \cref{asm:lip} \ref{asm:lip:rho} implies that $\rho_{\theta,t}(X)$ is jointly $L_\rho$-Lipschitz continuous on $\CAL V_{\CAL S\times\Theta}\times\Theta$ according to \cite[Remark 2]{hinderer2005lipschitz}.

Finally, since $\bar v\in\CAL C_{L_0}(\CAL S\times\Theta)$, we have that $\rho_{\theta,t}(\bar v(s,\theta))$ is $L_0L_\rho$-Lipschitz continuous on $\CAL S\times\Theta$ by \cite[Lemma 2,1 (b)]{hinderer2005lipschitz}, and the proof is complete.

\section{Proof of Lemma \ref{lem:sigma_lip}}\label{pf:sigma_lip}
Let $\psi,\psi'\in\CAL M(\CAL S)$ and pick any $\theta,\theta'\in\Theta$. Since $\bar v\in\CAL C_{L_0}(\CAL S\times\Theta)$, we have
\begin{align*}
    & \abs{\int \bar v(\tilde s,\theta)\psi(d\tilde s) - \int v(\tilde s,\theta')\psi'(d\tilde s)} \\
    &\qquad \leq \abs{\int \bar  v(\tilde s,\theta)\psi(d\tilde s) - \int \bar v(\tilde s,\theta)\psi'(d\tilde s)}+ \int \abs{ \bar v(\tilde s,\theta) - \bar v(\tilde s,\theta')}\psi'(d\tilde s)\\
    &\qquad \leq L_0 \Big(W_1(\psi,\psi') + d_{\Theta}(\theta,\theta')\Big),
\end{align*}
which shows that $(\psi,\theta)\mapsto\int\bar v(\tilde s,\theta)\psi(d\tilde s)$ is $L_0$-Lipschitz continuous. Due to Assumption \ref{asm:lip:rho} and an application of \cite[Lemma 3.2]{hinderer2005lipschitz}, we have
\begin{align*}
    & \abs{\sup_{\psi\in \Psi_t(s,a,\theta)} \int \bar v(\tilde s,\theta) \psi(d\tilde s)  - \sup_{\psi\in \Psi_t(s',a',\theta')} \int \bar v(\tilde s,\theta') \psi(d\tilde s) } \\
    &\qquad \leq (1+L_{\rho}) L_0 \Big(d_{\CAL S}(s, s')+d_{\CAL A}(a,a') + d_\Theta(\theta,\theta') \Big).
\end{align*}
This directly yields that $\sigma_{q,\bar v}$ is $ L_0(1+L_{\rho})$-Lipschitz continuous for any $\bar v\in\CAL C_{L_0}(\CAL S\times\Theta)$, which completes the proof.

\section{Proof of Lemma \ref{lem:phi_cts}}\label{pf:phi_cts}
First, we prove the upper hemicontinuity of the correspondence $\Phi$. Consider a sequence of triples $\{(s_n,m_n,\theta_n)\}_{n\in\Na}\subset\CAL S\times\CAL L_p^*(\CAL S,\CAL B(\CAL S),\mathbb{Q})\times\Theta$ such that $s_n\to s\in\CAL S$, $m_n\to m\in\CAL L_p^*(\CAL S,\CAL B(\CAL S),\mathbb{Q})$ in the weak* sense, and $\theta_n\to\theta\in\Theta$. For any $\{\phi_n\}_{n\in\Na}$ satisfying $\phi_n\in \Phi(s_n,m_n,\theta_n)$, we have $\phi_n=h(s_n,m_n,\theta_n)m_n$ by \cref{eq:Phi_smth}. In this case, for any bounded measurable function $f:\CAL S\to\Re$, we have
\begin{align}\label{eq:fh_bdd}
    \abs{fh(s_n,m_n,\theta_n)m_n}& \leq \norm{f}_\infty\phi_n \text{ and }
    \int f\bar\phi(s')\BBM Q(ds')&\leq \int \norm{f}_\infty\bar\phi(s')\BBM Q(ds')<\infty.
\end{align}
Therefore, if $\phi_n\to\phi$ in the weak* sense (which yields $\phi_n \BBM Q\sw \phi\, \BBM Q$), then
\begin{align*}
    \int f(s')h(s,m,\theta)m(s')\BBM Q(ds')\stackrel{(a)}{=}&\lim_{n\to\infty}\int f(s')h(s_n,m_n,\theta)m_n(s')\BBM Q(ds')\\
    =&\lim_{n\to\infty}\int f(s') \phi_n(s')\BBM Q(ds')\stackrel{(b)}{=}\int f(s')\phi(s')\BBM Q(ds'),
\end{align*}
where (a) holds due to the Dominated Convergence Theorem, and (b) holds by the definition of setwise convergence. In this case, $\phi=h(s,m,\theta)m$, implying $\phi\in\Phi(s,m,\theta)$ by \cref{eq:Phi_smth}, which shows that $\Phi$ is upper hemicontinuous at $(s,m,\theta)$.

We next show that $\Phi$ is lower hemicontinuous. Again consider a sequence of triples $\{(s_n,m_n,\theta_n)\}_{n\in\Na}\subset\CAL S\times\CAL L_p^*(\CAL S,\CAL B(\CAL S),\mathbb{Q})\times\Theta$ such that: $s_n\to s\in\CAL S$, $m_n\to m\in\CAL L_p^*(\CAL S,\CAL B(\CAL S),\mathbb{Q})$ pointwise, and $\theta_n\to\theta\in\Theta$. For any $\phi\in \Phi(s,m,\theta)$, there exists $h\in\CAL H'$ such that $\phi=h(s,m,\theta)m$. We need to prove that there exists $\phi_n\in\Phi(s_n,m_n,\theta_n)$ such that $\phi_n\to\phi$ in the weak* sense. Towards this end, pick $\phi_n=h(s_n,m_n,\theta_n)m_n$, , then by \cref{eq:Phi_smth}, $\phi_n\in\Phi(s_n,m_n,\theta_n)$. Thus, by the Dominated Convergence Theorem with \cref{eq:fh_bdd} again, we have
\begin{align*}
    \int f(s')\phi_n(s')\BBM Q(ds')=&\int f(s')h(s_n,m_n,\theta_n)m_n(s')\BBM Q(ds')\\
    \to&\int f(s')h(s,m,\theta)m(s')\BBM Q(ds')=\int f(s')\phi(s')\BBM Q(ds'),
\end{align*}
which yields $\phi_n\BBM Q\sw\phi\, \BBM Q$, and thus, $\phi_n\to\phi$ in the weak* sense. Then $\Phi$ is lower hemicontinuous at $(s,m,\theta)$. Since $\Phi$ is both upper/lower hemicontinuous, and $(s,m,\theta)$ is arbitrary, we can conclude that $\Phi$ is a continuous correspondence.

\section{Proof of Claims in Example \ref{exm:cvar_cts}}\label{pf:cvar_cts}
By \cref{eq:VaR}, we have $\text{VaR}_u(X_n)=F_{X_n}^{-1}(u)=\inf\{x:\pr{X_n\geq x}< 1-u\}$. Then by the Markov inequality, if $X_n\geq 0$ then
\begin{align*}
    \pr{X_n\geq x}= \pr{X_n^p\geq x^p}\leq \frac{\ex{X_n^p}}{x^p},\;\forall x>0.
\end{align*}
For any $u\in[0,1)$ and $p>1$, we see
\begin{align*}
    \{x:\pr{X_n\geq x}<1-u\}\supset \{x:\ex{X_n^p} < (1-u)x^p\}.
\end{align*}
Taking the infimum over $x$ on both sides, we have
\begin{align*}
    F_{X_n}^{-1}(u)\leq \inf\{x:\ex{X_n^p}<(1-u)x^p\}=\left(\frac{\ex{X_n^p}}{1-u}\right)^{\frac{1}{p}}.
\end{align*}
Pick
\begin{align}\label{eq:var_dom}
    \bar f_n(u)=\frac{\ex{X_n^p}^{\frac{1}{p}}}{(1-u)^{\frac{1}{p}}}\ind{u\in[\lambda(\theta_n),1]},
\end{align}
then $F_{X_n}^{-1}(u)\leq \bar f_n(u)$ for all $u\in[0,1)$ and $n\in\Na$. Moreover,
\begin{align}
    \lim_{n\to\infty}\int_0^1\bar f_n(u)du =& \lim_{n\to\infty}\ex{X_n^p}^{\frac{1}{p}}\int_0^1 \frac{\ind{u\in[\lambda(\theta_n),1]}}{(1-u)^{\frac{1}{p}}}du\nonumber\\
    =& \lim_{n\to\infty}\ex{X_n^p}^{\frac{1}{p}} \frac{p(1-\lambda(\theta_n))^{1-\frac{1}{p}}}{p-1} \nonumber\\
    =& \ex{X^p}^{\frac{1}{p}}\frac{p(1-\lambda(\theta))^{1-\frac{1}{p}}}{p-1}=\int_0^1 \bar f(u) du<\infty, \label{eq:bar_f_dct}
\end{align}
since $X_n\cp X$, and $\lambda:\Theta\to[0,1)$ is continuous and uniformly bounded. Furthermore, as we stated in \cref{exm:inf}, $X_n\cp X$ implies $F_{X_n}\to F_X$ and thus \cite[Proposition 5, p.250]{fristedt2013modern} implies that $F_{X_n}^{-1}\to F_X^{-1}$ pointwise\footnote{In \cite[Proposition 5, p.250]{fristedt2013modern}, the assumption of the convergence of probability measures is indeed the convergence of distribution functions. See \cite[Definition 1, p.244]{fristedt2013modern} for details.}. In this case, we have $\abs{F_{X_n}^{-1}(u)}\leq \bar f_n(u)+M$ for all $n\in\Na$, which yields
\begin{align*}
    &\int_{\lambda(\theta_n)}^1 \text{VaR}_u(X_n) du = \int_0^1 \ind{u\in[\lambda(\theta_n),1]} \text{VaR}_u(X_n) du\\
    &\qquad\stackrel{(a)}{\to} \int_0^1 \ind{u\in[\lambda(\theta),1]} \text{VaR}_u(X) du = \int_{\lambda(\theta)}^1 \text{VaR}_u(X) du,
\end{align*}
where (a) holds by the generalized Dominated Convergence Theorem. This establishes that $\text{CVaR}_{\lambda(\theta)}(X)$ is jointly continuous on $\CAL V\times\Theta$.


\section*{Acknowledgments}
Shiping Shao and Abhishek Gupta would like to acknowledge Ford Motor Company for supporting this research through a University Alliance Project.

\bibliographystyle{siamplain}
\bibliography{reference}

\begin{thebibliography}{10}

\bibitem{Aliprantis2006}
{\sc C.~D. Aliprantis and K.~C. Border}, {\em Infinite Dimensional Analysis},
  Springer-Verlag, 2006, \url{https://doi.org/10.1007/3-540-29587-9}.

\bibitem{armbruster2015decision}
{\sc B.~Armbruster and E.~Delage}, {\em Decision making under uncertainty when
  preference information is incomplete}, Management science, 61 (2015),
  pp.~111--128.

\bibitem{artzner1999coherent}
{\sc P.~Artzner, F.~Delbaen, J.-M. Eber, and D.~Heath}, {\em Coherent measures
  of risk}, Mathematical finance, 9 (1999), pp.~203--228.

\bibitem{bauerle2009dynamic}
{\sc N.~B{\"a}uerle and A.~Mundt}, {\em Dynamic mean-risk optimization in a
  binomial model}, Mathematical Methods of Operations Research, 70 (2009),
  pp.~219--239.

\bibitem{bauerle2011markov}
{\sc N.~B{\"a}uerle and J.~Ott}, {\em Markov decision processes with
  average-value-at-risk criteria}, Mathematical Methods of Operations Research,
  74 (2011), pp.~361--379.

\bibitem{bauerle2014more}
{\sc N.~B{\"a}uerle and U.~Rieder}, {\em More risk-sensitive {M}arkov decision
  processes}, Mathematics of Operations Research, 39 (2014), pp.~105--120.

\bibitem{billingsley1968convergence}
{\sc P.~Billingsley}, {\em Convergence of probability measures}, John Wiley,
  New York, 1968.

\bibitem{boda2004stochastic}
{\sc K.~Boda, J.~A. Filar, Y.~Lin, and L.~Spanjers}, {\em Stochastic target
  hitting time and the problem of early retirement}, IEEE Transactions on
  Automatic Control, 49 (2004), pp.~409--419.

\bibitem{bogachev2007measure}
{\sc V.~I. Bogachev and M.~A.~S. Ruas}, {\em Measure theory}, vol.~1, Springer,
  2007.

\bibitem{cao2003perturbation}
{\sc X.-R. Cao}, {\em From perturbation analysis to {M}arkov decision processes
  and reinforcement learning}, Discrete Event Dynamic Systems, 13 (2003),
  pp.~9--39.

\bibitem{chung1987discounted}
{\sc K.-J. Chung and M.~J. Sobel}, {\em Discounted {MDP}s: Distribution
  functions and exponential utility maximization}, SIAM journal on control and
  optimization, 25 (1987), pp.~49--62.

\bibitem{delage2018minimizing}
{\sc E.~Delage and J.~Y.-M. Li}, {\em Minimizing risk exposure when the choice
  of a risk measure is ambiguous}, Management Science, 64 (2018), pp.~327--344.

\bibitem{dentcheva2004optimality}
{\sc D.~Dentcheva and A.~Ruszczy{\'n}ski}, {\em Optimality and duality theory
  for stochastic optimization problems with nonlinear dominance constraints},
  Mathematical Programming, 99 (2004), pp.~329--350.

\bibitem{dutta1994parametric}
{\sc P.~K. Dutta, M.~K. Majumdar, and R.~K. Sundaram}, {\em Parametric
  continuity in dynamic programming problems}, Journal of Economic Dynamics and
  Control, 18 (1994), pp.~1069--1092.

\bibitem{feinberg2014berges}
{\sc E.~A. Feinberg, P.~O. Kasyanov, and M.~Voorneveld}, {\em Berge\'s maximum
  theorem for noncompact image sets}, Journal of Mathematical Analysis and
  Applications, 413 (2014), pp.~1040--1046.

\bibitem{fristedt2013modern}
{\sc B.~E. Fristedt and L.~F. Gray}, {\em A modern approach to probability
  theory}, Springer Science \& Business Media, 2013.

\bibitem{furukawa1972markovian}
{\sc N.~Furukawa}, {\em Markovian decision processes with compact action
  spaces}, The Annals of Mathematical Statistics, 43 (1972), pp.~1612--1622.

\bibitem{hernandez2012discrete}
{\sc O.~Hern{\'a}ndez-Lerma and J.~B. Lasserre}, {\em Discrete-time Markov
  control processes: basic optimality criteria}, vol.~30, Springer Science \&
  Business Media, 2012.

\bibitem{hinderer2005lipschitz}
{\sc K.~Hinderer}, {\em Lipschitz continuity of value functions in {M}arkovian
  decision processes}, Mathematical Methods of Operations Research, 62 (2005),
  pp.~3--22.

\bibitem{howard1972risk}
{\sc R.~A. Howard and J.~E. Matheson}, {\em Risk-sensitive {M}arkov decision
  processes}, Management science, 18 (1972), pp.~356--369.

\bibitem{inoue2003worst}
{\sc A.~Inoue}, {\em On the worst conditional expectation}, Journal of
  Mathematical Analysis and Applications, 286 (2003), pp.~237--247.

\bibitem{jaquette1973markov}
{\sc S.~C. Jaquette}, {\em Markov decision processes with a new optimality
  criterion: {D}iscrete time}, The Annals of Statistics, 1 (1973),
  pp.~496--505.

\bibitem{jaquette1976utility}
{\sc S.~C. Jaquette}, {\em A utility criterion for {M}arkov decision
  processes}, Management Science, 23 (1976), pp.~43--49.

\bibitem{jordan1977continuity}
{\sc J.~S. Jordan}, {\em The continuity of optimal dynamic decision rules},
  Econometrica: Journal of the Econometric Society,  (1977), pp.~1365--1376.

\bibitem{li2000optimal}
{\sc D.~Li and W.-L. Ng}, {\em Optimal dynamic portfolio selection:
  {M}ultiperiod mean-variance formulation}, Mathematical finance, 10 (2000),
  pp.~387--406.

\bibitem{li2021fitted}
{\sc H.~Li, S.~Shao, and A.~Gupta}, {\em Fitted value iteration in continuous
  {MDP}s with state dependent action sets}, IEEE Control Systems Letters, 6
  (2021), pp.~1310--1315.

\bibitem{maitra1968discounted}
{\sc A.~Maitra}, {\em Discounted dynamic programming on compact metric spaces},
  Sankhy{\=a}: The Indian Journal of Statistics, Series A,  (1968),
  pp.~211--216.

\bibitem{Markowitz1952portfolio}
{\sc H.~Markowitz}, {\em Portfolio selection}, The Journal of Finance, 7
  (1952), pp.~77--91, \url{http://www.jstor.org/stable/2975974} (accessed
  2022-06-30).

\bibitem{ogryczak1999stochastic}
{\sc W.~Ogryczak and A.~Ruszczy{\'n}ski}, {\em From stochastic dominance to
  mean-risk models: {S}emideviations as risk measures}, European journal of
  operational research, 116 (1999), pp.~33--50.

\bibitem{porteus1975optimality}
{\sc E.~L. Porteus}, {\em On the optimality of structured policies in countable
  stage decision processes}, Management Science, 22 (1975), pp.~148--157.

\bibitem{rockafellar2002deviation}
{\sc R.~T. Rockafellar, S.~P. Uryasev, and M.~Zabarankin}, {\em Deviation
  measures in risk analysis and optimization}, University of Florida,
  Department of Industrial \& Systems Engineering Working Paper,  (2002).

\bibitem{ruszczynski2010risk}
{\sc A.~Ruszczy{\'n}ski}, {\em Risk-averse dynamic programming for {M}arkov
  decision processes}, Mathematical programming, 125 (2010), pp.~235--261.

\bibitem{ruszczynski2014erratum}
{\sc A.~Ruszczy{\'n}ski}, {\em Erratum to: {R}isk-averse dynamic programming
  for {M}arkov decision processes}, Mathematical Programming, 145 (2014),
  pp.~601--604.

\bibitem{ruszczynski2006optimization}
{\sc A.~Ruszczy{\'n}ski and A.~Shapiro}, {\em Optimization of convex risk
  functions}, Mathematics of operations research, 31 (2006), pp.~433--452.

\bibitem{shapiro2013kusuoka}
{\sc A.~Shapiro}, {\em On {K}usuoka representation of law invariant risk
  measures}, Mathematics of Operations Research, 38 (2013), pp.~142--152.

\bibitem{shapiro2021lectures}
{\sc A.~Shapiro, D.~Dentcheva, and A.~Ruszczynski}, {\em Lectures on stochastic
  programming: modeling and theory}, SIAM, 2021.

\bibitem{stigum1969competitive}
{\sc B.~P. Stigum}, {\em Competitive equilibria under uncertainty}, The
  Quarterly Journal of Economics, 83 (1969), pp.~533--561.

\bibitem{stokey1989recursive}
{\sc N.~L. Stokey}, {\em Recursive methods in economic dynamics}, Harvard
  University Press, 1989.

\bibitem{wu1999minimizing}
{\sc C.~Wu and Y.~Lin}, {\em Minimizing risk models in {M}arkov decision
  processes with policies depending on target values}, Journal of mathematical
  analysis and applications, 231 (1999), pp.~47--67.

\end{thebibliography}

\end{document}